\documentclass{amsart}

\usepackage{amsmath}
\usepackage{amsthm}
\usepackage{hyperref}
\usepackage{ragged2e}
\usepackage{enumitem}
\usepackage{tikz-cd}
\usepackage{mathtools}

\usepackage{amssymb}
\newcommand{\bb}{\mathbb}

\newcommand{\Gal}{\operatorname{Gal}}
\newcommand{\Lap}{\triangle}

\newcommand{\Lip}{\operatorname{Lip}}
\newcommand{\an}{\operatorname{an}}
\newcommand{\vphi}{\varphi}

\newcommand{\sph}{\operatorname{sph}}
\newcommand{\Diag}{\operatorname{Diag}}
\newcommand{\bbf}{\mathbf}

\newcommand{\ovl}{\overline}

\newcommand{\berkP}{\mathbb{P}^1_{\operatorname{Berk},v}}
\newcommand{\berkA}{\mathbb{A}^1_{\operatorname{Berk},v}}

\newcommand{\cal}{\mathcal}
\newcommand{\Spec}{\operatorname{Spec}}
\newcommand{\ord}{\operatorname{ord}}

\newcommand{\fin}{\operatorname{fin}}

\newcommand{\tor}{\operatorname{tor}}
\newcommand{\eps}{\epsilon}

\newcommand{\supp}{\operatorname{supp}}

\newtheorem{theorem}{Theorem}
\numberwithin{theorem}{section}

\newtheorem{proposition}[theorem]{Proposition}

\numberwithin{example}{section}

\numberwithin{definition}{section}
\newtheorem{corollary}[theorem]{Corollary}

\newtheorem{conjecture}[theorem]{Conjecture}
\usepackage[style=ext-alphabetic,articlein=false]{biblatex}
\makeindex
\title[Uniform Bounds on S-Integral Torsion Points]{Uniform Bounds on S-Integral Torsion Points for $\bb{G}_m$ and Elliptic Curves}
\author{Jit Wu Yap}
\begin{document}

\begin{abstract}
    Let $K$ be a number field, $S$ a finite set of places. For $\bb{G}_m$ or an elliptic curve $E$ defined over $K$, we establish uniformity results on the number of $S$-integral torsion points relative to a non-torsion point $\beta$, as $\beta$ varies over number fields of bounded degree. In particular for $\bb{G}_m$, if $D$ is a positive integer, we prove a uniform bound on the degree of a torsion point $\zeta$ that is $S$-integral relative to a non-torsion point $\beta$ with degree $\leq D$. 
\end{abstract}

\maketitle

\tableofcontents

\section{Introduction}
\subsection{Statements and Results}
Let $K$ be a number field and $S$ a finite set of places of $K$ that includes all archimedean places. Given two points $\alpha,\beta: \Spec \ovl{K} \to \bb{P}^1_K$, we say that they are $S$-integral relative to each other if their Zariski closures in $\bb{P}^1_{O_K}$ do not intersect outside $S$. Let $\mu_{\infty}$ denote the set of all roots of unity $\zeta$, where $\zeta^n = 1$ for some $n$. In \cite{BIR08}, Baker, Ih and Rumely prove the following result.

\begin{theorem}[Theorem 2.1, \cite{BIR08}] \label{IntroBIR1}
Let $K$ be a number field and $S$ a finite set of places including all archimedean places. For each $\beta \in \ovl{K}^{\times} \setminus \mu_{\infty}$, the set of $\zeta \in \mu_{\infty}$ such that $\zeta$ is $S$-integral relative to $\beta$ is finite.
\end{theorem}

As observed in \cite{BIR08}, it is necessary that $\beta$ not be a root of unity. Our first result is a uniform version of Theorem \ref{IntroBIR1}. For an algebraic number $x \in \ovl{\bb{Q}}$ and a number field $K$, we let $\deg_K(x) = |\Gal(\ovl{K}/K) \cdot x|$ denote the size of the $\Gal(\ovl{K}/K)$-orbit of $x$. 

\begin{theorem} \label{IntroUniformIh1}
Let $S$ be a finite set of places of $\bb{Q}$ including all archimedean places and $K$ a number field. Then there exists a constant $C = C([K:\bb{Q}],S)$ such that the following holds: for any $\beta \in K^{\times} \setminus \mu_{\infty}$, if $\zeta \in \mu_{\infty}$ is $S$-integral relative to $\beta$ then $\deg_{\bb{Q}}(\zeta) < C$.
\end{theorem}

Thus for a fixed number field $K$, there are only finitely many roots of unity $\zeta$ that can be $S$-integral relative to any $\beta \in K^* \setminus \mu_{\infty}$. In fact, our statement is stronger as we allow $\beta$ to vary over number fields $K$ with $[K:\bb{Q}] \leq D$ for some positive integer $D \geq 1$, which is why in our statement we let $S$ be a finite set of places of $\bb{Q}$ rather than $K$.
\par 
In \cite{GI13}, Grant and Ih show how one can deduce using results of Schinzel a uniform bound on the order of a root of unity $\zeta$ that is $S$-integral relative to a $S$-unit $\beta \in K^* \setminus \mu_{\infty}$. Here, our results hold without any condition on $|\beta|_v$. 
\par 
A natural question to ask is how $C$ depends on $[K:\bb{Q}]$. In \cite{Sil95}, a construction of Boyd is given which produces a sequence $\alpha_n \in \ovl{\bb{Q}}$ such that $\alpha_n$ is $S$-integral relative to some root of unity of order $> \deg_{\bb{Q}}(\alpha_n)^{O(1/\log \log \deg_{\bb{Q}}(\alpha_n))}$, with $\deg_{\bb{Q}}(\alpha_n)$ tending to infinity. Hence one might expect that $C$ should grow polynomially with $[K:\bb{Q}]$. Our proof of Theorem \ref{IntroUniformIh1} only provides an upper bound that grows exponentially with $[K:\bb{Q}]$, but if one allows a small set of exceptions for each $\beta$, it is possible to get a polynomial bound. We let $S_{\fin}$ denote the subset of $S$ consisting exactly all non-archimedean places.

\begin{theorem} \label{IntroUniformIh3}
Let $K$ be a number field and $S$ be a finite set of places of $K$ including all archimedean places. Then there exists a constant $c > 0$, independent of $K$ and $S$, such that for all $\beta \in K^{\times} \setminus \mu_{\infty}$, the set
$$\{\zeta \in \mu_{\infty} : \deg_{K}(\zeta) > c |S|^3 [K:\bb{Q}]^{6} \text{ and } \zeta \text{ is S-integral relative to } \beta\}$$
is a union of at most $|S_{\fin}|$ $\Gal(\ovl{K}/K)$-orbits where $S_{\fin}$ is the susbet of $S$ consisting exactly of all finite places.
\end{theorem}

For example, if $S = \{2,3,\infty\}$, then for each $\beta \in \ovl{K}^{\times} \setminus \mu_{\infty}$ with $[\bb{Q}(\beta):\bb{Q}] \leq D$, except for $2D$ possible $\Gal(\ovl{\bb{Q}}/\bb{Q})$-orbits as exceptions, all other roots of unity $\zeta$ that are $S$-integral relative to $\beta$ must have $< cD^{6}$ conjugates, where $c > 0$ is a constant independent of $D$ and $\beta$. Here, the $2D$ comes from taking all the places in $\bb{Q}(\beta)$ that live above $S$. We may view these exceptions in the same way exceptions arise from Thue--Siegel--Roth. 
\par 
Baker--Ih--Rumely also prove an analogous result for elliptic curves $E$, where $\mu_{\infty}$ is replaced by the set of all torsion points of $E$, which we denote by $E_{\tor}$.. Our next theorem is the analogue of Theorem \ref{IntroUniformIh3} for elliptic curves $E$ with complex multiplication. 

\begin{theorem} \label{IntroUniformIh2}
Let $K$ be a number field and $S$ be a finite set of places of $K$ including all archimedean places. Let $E$ be an elliptic curve with complex multiplication defined over $K$ and let $L$ be a finite extension of $K$. Then there exists a constant $C = C([L:K],S)$ such that the following holds: if $\beta \in E(L) \setminus E_{\tor}$ is $S$-integral relative to some $z \in E_{\tor}$, then $\deg_K(z) < C$. 
\end{theorem}

The assumption that our elliptic curve has complex multiplication is needed in our proof of Theorem \ref{IntroUniformIh2} for technical reasons. For our analogue of Theorem \ref{IntroUniformIh3} where we allow finitely many exceptions, we are able to prove a version that holds for all elliptic curves $E$ without any assumption of complex multiplication. 

\begin{theorem} \label{IntroUniformIh4}
Let $K$ be a number field and $E$ an elliptic curve defined over $K$ that is semistable. Let $L$ be a finite extension of $K$ and $S$ a finite set of places of $L$ including all archimedean places. Then there exists a constant $c = c(E,K) > 0$, independent of $L$, such that for all $\beta \in E(L) \setminus E_{\tor}$, the set
$$\{ z \in E_{\tor} \mid \deg_{K}(x) > c |S|^3 [L:K]^{11} \text{ and } z \text{ is S-integral relative to } \beta \}$$
is a union of at most $|S|$ $\Gal(\ovl{K}/L)$-orbits.
\end{theorem}

Here, we have to allow an exception for every archimedean place too and not just non-archimedean places due to the upper bounds coming from linear forms in elliptic logarithms being not strong enough. We also have to assume that $E$ has semistable reduction over $K$, but that is not so important as every elliptic curve achieves semistable reduction after base changing to a finite extension.

\subsection{Overview of Proofs} 
The main idea is the notion of logarithmic equidistribution. We first explain the case of $\bb{G}_m$. Let $K$ be a number field and let $h(x)$ denote the usual logarithmic Weil height on $\bb{P}^1(\ovl{K})$. Then if we let $\mu$ denote the uniform probability measure on the unit circle, it is known by Bilu \cite{Bil97} that if $(x_n)$ is a distinct sequence of elements of $\ovl{K}$ with $h(x_n) \to 0$, then the Galois orbits of $x_n$ equidistribute to $\mu$. More precisely, if one lets $F_n$ be the Galois orbit of $x_n$ over $K$, then for any continuous function $f: \bb{P}^1(\bb{C}) \to \bb{R}$, we have
\begin{equation} \label{eq: Equib1}
\lim_{n \to \infty} \frac{1}{|F_n|} \sum_{x \in F_n} f(x) \to \int f d \mu.
\end{equation}
Now fix an archimedean place $v$ of $K$ along with an extension to $\ovl{K}$. This gives rise to an embedding of $\ovl{K} \xrightarrow{} \bb{C}$. The strategy of Baker--Ih--Rumely \cite{BIR08} is to prove that \eqref{eq: Equib1} also holds for the function $\log |x-\beta|_v$, which is not continuous as it has a logarithmic pole at $\beta$, when $(x_n)$ is a sequence of torsion points. This is achieved by proving that not any point of $F_n$ is too close to $\beta$ using linear forms in logarithms, and then proving a quantitative estimate for the convergence in \eqref{eq: Equib1}. 
\par 
For elliptic curves $E$ that are defined over $K$, there is a similar statement proven by Szpiro--Ullmo--Zhang \cite{SUZ97}. Here, one has to fix an archimedean place $v$ of $K$ along with an extension to $\ovl{K}$, which gives an embedding $\ovl{K} \xhookrightarrow{} \bb{C}$. One then replaces $\mu$ with the Haar measure on $E(\bb{C}_v)$ and the Weil height with the Neron--Tate height $h_E$ on $E(\ovl{K})$. The function $\log |x-\beta_v$ is replaced by $\lambda_v(x-\beta)$, where $\lambda_v$ is the Neron local height function.
\par 
It will be necessary to also consider analogous statements for non-archimedean places $v$ too. This requires us to work on the Berkovich analytification $\berkP$ and $E^{\an}_v$. For $\bb{G}_m$, it is known independently by the works of Baker--Rumely \cite{BR06}, Favre--Rivera-Letelier \cite{FRL06} and Chambert-Loir \cite{CL06} that our Galois orbits $F_n$ converge to a delta mass at the Gauss point $\delta_{\zeta(0,1)}$. We refer readers to \cite{BR10} for background on the Berkovich projective line. 
\par 
For elliptic curves, it was proven by Baker--Petsche \cite{BP05} that there exists a uniform probability measure $\mu_v$ on $E^{\an}_v$ such that the Galois orbits $F_n$ converge to the said measure. 
\par 
Although \cite{BIR08} provides quantitative estimates for the rate of convergence, it is important for Theorems \ref{IntroUniformIh3} and \ref{IntroUniformIh4} to understand how these estimates depend as the base field $K$ change, which are not covered by the results of \cite{BIR08}. To handle this issue for $\bb{G}_m$, we instead use Favre--Rivera-Letelier's quantitative equidistribution theorem which allows us to understand the rate of convergence even as the base field $K$ varies.
\par 
For elliptic curves, Baker--Petsche \cite{BP05} gives also a quantitative estimate on the convergence, but only in terms of an upper bound on a local discrepancy term $\Lambda_v(F)$. When $v$ is non-archimedean, Petsche \cite{Pet09} obtains an explicit estimate for the convergence in \eqref{eq: Equib1} in terms of the local discrepancy $\Lambda_v(F)$ (see Proposition \ref{QuantNonArchElliptic1}. For archimedean $v$, it is not quite clear how to directly obtain an explicit estimate from the local discrepancy. We instead follow the approach in \cite{FRL06} to obtain an explicit estimate from an upper bound of the local discrepancy term (see Proposition \ref{ArchQuant1}). 
\par 
To obtain results that are uniform in $\beta$, the main obstacle is to bound how close a torsion point $x$ can be to $\beta$ in terms of the size of its Galois orbit $F$ and the height $h(\beta)$. For uniformity, our approach requires that our bound is linear with respect to $h(\beta)$. Such a bound can be obtained using linear forms in logarithms. For $\bb{G}_m$, we use due to Laurent, Mignotte and Nesterenko \cite{LMN95} which give us the bound we want. For non-archimedean $v$, the bounds arising from $p$-adic linear forms in logarithms have a factor of $p^D$ where $D = [\bb{Q}(\beta):\bb{Q}]$, which gives us an exponential bound in $D$ instead of a polynomial one. To achieve a polynomial bound, we use a simple discreteness property that roots of unity satisfy in non-archimedean places, which leads us to allowing one exception for each finite place $v$. 
\par 
In the case of elliptic curves, there is a theory of linear forms in elliptic logarithms. Using results of David--Hirata-Kohno \cite{SH09}, it turns out that in general, the bounds obtained are not strong enough to prove a version of Theorem \ref{IntroUniformIh1} for elliptic curves as there is an extra factor of a power of $\log^+ h(\beta)$ that appears in our upper bound on $\lambda_v(x-\beta)$. This forces us to allow the possible existence of an exceptional Galois orbit $F$ for each place that could be $S$-integral relative to $\beta$ in order to strengthen our upper bound. For elliptic curves with complex multiplication, Ably and Gaudron \cite{AG03} have proven stronger bounds where there is no extra $\log^+ h(\beta)$ factor, which allows us to obtain Theorem \ref{IntroUniformIh2} for CM elliptic curves. The non-archimedean place is more technical compared to that of $z^d$ for places $v$ where the elliptic curve $E$ has bad reduction, but one can still show some form of discreteness and we are able to conclude our uniform bounds.
\par 
Let us mention that there are analogous results of \cite{BIR08} in the setting of Drinfeld modules \cite{Dra14}. We expect that the methods here, combined with the equidistribution theorem proven in \cite{Dra13}, would lead to uniform bounds too. We also would like to mention there is a general conjecture by Ih on the finiteness of $S$-integrality of preperiodic points. Given a rational map $\vphi: \bb{P}^1 \to \bb{P}^1$, we say that a point $x$ is preperiodic if $\vphi^m(x) = \vphi^n(x)$ for some distinct natural numbers $m,n$.
\begin{conjecture} \label{Ih1} (Ih's Conjecture) 
Let $\vphi: \bb{P}^1 \to \bb{P}^1$ be a rational map of degree $d \geq 2$ defined over a number field $K$ and $S$ a finite set of places of $K$. For any non-preperiodic point $\beta \in \bb{P}^1(K)$, there are only finitely many preperiodic points $x \in \bb{P}^1(\ovl{K})$ that are $S$-integral relative to $\beta$.
\end{conjecture}

Ih's conjecture has been proven for power and Lattès maps by Baker--Ih--Rumely \cite{BIR08}, and for Chebyshev Maps by Ih-Tucker \cite{IT10}. Petsche \cite{Pet08} has established the conjecture for any rational map $\vphi$ of degree $d \geq 2$, but under the additional assumption that $\beta$ is a totally Fatou point. Quantitative results in Petsche's setting have recently been obtained by Young \cite{You22}.
\par 
We would also like to mention a result of Ingram \cite{Ing09} regarding the number of $n$ such that $[n]P$ is an integral point where $P \in E(\bb{Q})$ is a non-torsion point. Ingram proves that for $n$ larger than some bound depending on the Tamagawa number on $E$, there is at most one possible value of $n$ such that $[n]P$ is integral, similar to our results in Theorem \ref{IntroUniformIh4}.

\subsection{Acknowledgements} The author would like to thank our advisor Laura DeMarco, along with Niki Myrto Mavraki and Yan Sheng Ang for helpful discussions about the paper and the problem. The author would like to thank Paul Fili, Dragos Ghioca and Marley Young for helpful comments. The author would also like to thank the referee for numerous comments and suggestions.

\section{Quantitative Equidistribution on $\bb{P}^1$}
We now recall the quantitative version of equidistribution as proven by Favre--Rivera-Letelier in \cite{FRL06}. Such quantitative equidistribution methods have been used in \cite{Fil17}, \cite{DKY20}, \cite{DKY21} and \cite{You22} to obtain uniform results. We first introduce the Berkovich projective line $\berkP$.

\subsection{The Berkovich Projective Line $\berkP$} Let $\bb{C}_v$ be an algebraically closed field that is complete with respect to a non-archimedean valuation $v$. When attempting to study dynamics of a rational map $\vphi: \bb{P}^1(\bb{C}_v) \to \bb{P}^1(\bb{C}_v)$, it turns out that it is more natural to study the situation over the Berkovich analytification $\vphi: \berkP \to \berkP$. Interested readers may consult \cite{Ben19} for a thorough treatment of the theory. 
\par 
Let $D(a,r) \subseteq \bb{A}^1(\bb{C}_v)$ be the open disc that is centered at $a$ with radius $r$. Each such disc corresponds to a point on $\berkP$. When $r = 0$, we can identify this with the usual point $a \in \bb{P}^1(\bb{C}_v)$. This are known as the Type I points, or also as the classical points. Let $\zeta(a,r)$ denote the point on $\berkP$ corresponding to the disc $D(a,r)$. When $r$ is an element of the value group $|\bb{C}_v|^{\times}$, we say that $\zeta(a,r)$ is a Type II point. If it is not, we say that $\zeta(a,r)$ is a Type III point. The remaining points of $\berkA$ are called Type IV points, and they correspond to a nested intersection of discs $\cdots \subseteq D_n \subseteq \cdots \subseteq D_1$ such that $\cap_{n=1}^{\infty} D_n = \emptyset$ but their radii do not go to zero. Finally, $\berkP$ consists of $\berkA$ along with an extra type I point called $\infty$. 
\par 
Given any element of $f(z) \in \bb{C}_v[z]$, the function $z \mapsto |f(z)|_v$ extends to a function on $\berkA \to \bb{R}_{\geq 0}$ and the topology on $\berkA$ is given the weakest topology such that all polynomials $f(z)$ are continuous on $\berkA$. This makes $\berkP$ into a compact and path-connected space where $\bb{P}^1(\bb{C}_v)$ sits inside as a dense subset. 
\par 
The Berkovich space allows one to develop a suitable analogue of the Laplacian $\Lap = \frac{1}{2\pi} (\frac{\partial^2}{\partial x^2} + \frac{\partial^2}{\partial y^2})$ for non-archimedean places. This is developed in full in Baker-Rumely \cite{BR10}. For a suitable class of continuous functions $f: \berkP \to \bb{R}$, its Laplacian $\Lap f$ is a signed Borel measure on $\berkP$ of total mass zero. Also, it is a self-adjoint operator in the sense that
$$\int f \Lap g = \int g \Lap f.$$
This fact will be repeatedly used later in computations. Here, we note that our Laplacian is normalized so that $\Lap \log |z| = \delta_{0} - \delta_{\infty}$. Using non-archimedean potential theory, it is then possible to construct an analogue of the equilibrium measure $\mu_{\vphi,v}$ for a given rational map $\vphi: \berkP \to \berkP$ of degree $d \geq 2$. For example, if $\vphi$ has good reduction, it is simply the delta mass at the Gauss point $\zeta(0,1)$.

\subsection{Quantitative Equidistribution following Favre--Rivera-Letelier }
Let $K$ be a number field, $K_v$ its completion for a place $v$ and $\bb{C}_v$ the completion of $\overline{K}_v$. Let $M_K$ denote the places of $v$. We first introduce some definitions from \cite{FRL06}. For each $v \in M_K$, let $\rho_v$ be a measure on $\berkP$.
\par 
We say that $\rho_v$ has continuous potentials if $\rho_v = \lambda_v + \Lap g$ for some continuous function $g: \berkP \to \bb{R}$ which is the difference of two continuous subharmonic functions. Here, $\lambda_v$ is the Dirac mass supported at the Gauss point $\delta_{\zeta(0,1)}$ for non-archimedean $v$ and is the uniform probability measure supported on the unit circle for archimedean $v$. If $d$ is a metric on $\bb{P}^1(\bb{C}_v)$, we say that $\rho_v$ has Hölder-continuous potentials with exponent $\kappa$ with respect to a metric $d$ if furthermore there exists a constant $C > 0$ such that
$$|g(z)-g(w)| \leq Cd(z,w)^{\kappa}$$
for all classical points $z,w \in \bb{P}^1(\bb{C}_v)$. 
\par 
A collection of measures $\rho = (\rho_v)_{v \in M_K}$ is said to be an adelic measure if each $\rho_v$ has continuous potentials and $\rho_v = \lambda_v$ for all but finitely many $v$. Given two measures $\rho_v,\rho_v'$ on $\berkP$, we define a bilinear form by
$$( \rho_v, \rho_v')_v = -\int_{\berkA \times \berkA \setminus \Diag} \log|z-w|_v d \rho_v(z) d \rho'_v(w)$$
where $\Diag = \{(z,z), z \in \bb{C}_v\}$ is the diagonal of classical points. This integral exists if both $\rho_v,\rho_v'$ either have continuous potentials or are discrete point masses. For a finite set of points $F \subset \overline{K}$ that is $\Gal(\overline{K}/K)$-invariant, we can then define its height with respect to $\rho$ as
$$h_\rho(F) = \frac{1}{2} \sum_{v \in M_K} (\!( [F]-\rho_v,[F]-\rho_v )\!)_v$$
where $[F] = \frac{1}{|F|} \sum_{x \in F} \delta_{x}$ and $(\!(\, \, , \, \,)\!) = N_v(\, \, , \, \,)$ with $N_v = [K_v:\bb{Q}_v]/[K:\bb{Q}]$. 
\par 
For an infinite place $v$, we say a continuous function $f: \bb{P}^1(\bb{C}_v)\to \bb{R}$ is of class $C^k_{\sph}$ if it is $C^k$ with respect to the spherical metric 
$$d_{\sph}(x,y)= \frac{|x_1 y_2 - x_2 y_1|_v}{\sqrt{|x_1|^2 + |y_1|^2} \sqrt{|y_1|^2 + |y_2|^2}}$$
where $x = [x_1:x_2]$ and $y = [y_1:y_2]$. For a finite place $v$, we say $f: \berkP \to \bb{R}$ is of class $C^k_{\sph}$ if it is locally constant outside of a finite subtree $T \subset \bb{H}_v$ and $T$ is a finite union of segments where $f$ is of the usual class $C^k$ on. Given $f$ of class $C^k_{\sph}$ for $k \geq 1$, we define 
$$\langle f , f \rangle_v =  \int_{\bb{C}}\left(\frac{\partial f}{\partial x} \right)^2 + \left(\frac{\partial f}{\partial y} \right)^2 dx dy$$
if $v$ is archimedean. If $v$ is non-archimedean, we fix a basepoint $S_0 \in \bb{H}_v$ and let $\partial f(S)$ be the derivative of $f$ restricted to the segment $[S_0,S]$. Then we define
$$ \langle f , f \rangle_v = \int_{\bbf{P}^1(\bb{C}_v)} (\partial f)^2 d \lambda.$$
We note that these definitions of energy still work if one assumes $f$ is merely continuous, if we use the weak derivative instead.
We can now state the quantitative equidistribution result of Favre--Rivera-Letelier. 

\begin{theorem}[Theorem 7, \cite{FRL06}]  \label{QuantEquib1} 
Let $\rho = \{\rho_v\}_{v \in M_K}$ be an adelic measure where each $\rho_v$  has Hölder-continuous potentials of exponent $\kappa \leq 1$ with respect to the spherical metric. Then there exists a constant $C > 0$, only depending on $\rho_v$'s and $K$, such that for all places $v$ and all functions $f$ of class $C^1_{\sph}$ on $\berkP$, and for all finite $\Gal(\overline{K}/K)$-invariant sets $F$, we have
$$\left| \frac{1}{|F|} \sum_{\alpha \in F} f(\alpha) - \int_{\berkP} f d\rho_v \right| \leq \frac{\Lip_{\sph}(f)}{|F|^{1/\kappa}} + \left(2 h_{\rho}(F) + C \frac{\log|F|}{|F|} \right)^{1/2} \langle f ,f \rangle_v^{1/2}.$$
Here, $\Lip_{\sph}$ is the Lipschitz constant for $f$ with respect to the spherical metric.
\end{theorem}

It will be important for us to find out the dependency of $C$ on $K$, while the $\rho_v$'s remains fixed. This happens for example when $\vphi: \bb{P}^1 \to \bb{P}^1$ is a rational map defined over a number field $K$ and we view it as a map over an extension $L$. We would also like to work wit the standard Euclidean metric, rather than the spherical metric. As such, we will follow the proof in \cite{FRL06} to obtain the following modified upper bound. We recall the regularization of measures given in \cite{FRL06} as follows. For an archimedean place $v$, we will let $[F]_{\eps}$ denote the measure $\frac{1}{|F|} \sum_{z \in F} \delta_{\zeta(z,\eps)}$ where $\delta_{\zeta(a,r)}$ denotes the uniform probability measure on the disc of radius $r$ centered at $a$. For a non-archimedean place, we will also let $[F]_{\eps}$ denote $\frac{1}{|F|} \sum_{z \in F} \delta_{\zeta(z,\eps)}$ where now $\delta_{\zeta(z,\eps)}$ denotes a delta mass at the Type II point represented by the disc of radius $\eps$ and centered at $z$. 

\begin{proposition} \label{QuantEquib3}
Let $\rho = \{\rho_v\}_{v \in M_K}$ be an adelic measure where each $\rho_v$  has Hölder-continuous potentials of exponent $\kappa \leq 1$ with respect to the spherical metric. Fix any $\delta > 0$. Then there exists a constant $C > 0$, only depending on $\rho_v$'s and $\delta$, such that for all places $v$ and all functions $f$ of class $C^1$ on $\berkP$, and for all finite $\Gal(\overline{K}/K)$-invariant sets $F$ not containing the point $\{\infty\}$, we have
$$\left| \frac{1}{|F|} \sum_{\alpha \in F} f(\alpha) - \int_{\berkP} f d\rho_v \right| \leq \frac{\Lip(f)}{|F|^{\delta/\kappa}} + \left(2 h_{\rho}(F) + C[K:\bb{Q}] \frac{\log|F|}{|F|} \right)^{1/2} \langle f ,f \rangle_v^{1/2}.$$
\end{proposition}

\begin{proof}
By \cite[Proposition 2.8]{FRL06} and \cite[Proposition 4.9]{FRL06}, there exists some constant $C' > 0$ independent of $\rho_v$ such that 
$$([F]- \rho_v, [F]-\rho_v) \geq ([F]_{\eps} - \rho_v, [F]_{\eps} - \rho_v) - 2\eta(\eps) - |F|^{-1}(C + \log \eps^{-1})$$
$$\geq 2 \eta(\eps) - |F|^{-1}(C + \log \eps^{-1}).$$
Here, we may take $\eta(\eps)$ to be a modulus of continuity for the spherical metric as the Euclidean metric is at most the spherical metric. Now if $S$ denotes the places of $M_K$ for which either $v$ is archimedean or $\rho_v$ is not the delta mass at the Gauss point, then for $v \not \in S$ we have 
$$([F]-\rho_v, [F]-\rho_v) \geq 0$$
without regularizing. Hence given a place $v_0$, we obtain
$$(([F]-\rho_{v_0}, [F]-\rho_{v_0})) \leq 2 h_{\rho}([F]) + |S|(2 \eta(\eps) + |F|^{-1}(C' + \log \eps^{-1}).$$
As $\eta(\eps)$ may be taken as $O(\eps^{\kappa})$, we set $\eps$ such that $\eps = |F|^{-\delta/\kappa}$. Hence for some other constant $C''$, now depending on the measures $\rho_v$, we obtain 
$$(([F]_{\eps}- \rho_{v_0}, [F]_{\eps}- \rho_{v_0})) \leq h_{\rho}([F]) + |S| C'' \frac{\log |F|}{|F|}.$$
Now given a test function $f$, we may apply Cauchy Schwarz on our positive definite form $(( \, , \ , ))$ to get
$$(\Lap f, [F]_{\eps} - \rho_{v_0}) \leq (\Lap f, \Lap f)^{1/2} \left( 2 h_{\rho}(F) + |S| C'' \frac{\log |F|}{|F|} \right)^{1/2}.$$
Since $(\Lap f, [F]_{\eps} - \rho_{v_0}) = \int f d([F]_{\eps} - \rho_{v_0})$, we obtain the same bound for $\int f d([F] - \rho_{v_0})$ after adding in an error term of $\eps \Lip(f) = \frac{\Lip(f)}{|F|^{\delta/\kappa}}$ as desired. Finally, if we vary our base field $K$, the constant $|S|$ changes by a factor of $[K:\bb{Q}]$. Hence we may replace $|S| C''$ by $C [K:\bb{Q}]$ as desired. 
\end{proof}

In the archimedean case, we wish to apply the theorem to functions that are not necessarily differentiable. Thus we have to extend the theorem to a larger class of test functions, which we do so by a continuity argument. For archimedean $v$, we consider Lipschitz continuous functions $f$ such that $\Lap f$ is a finite signed measure. We will let $\cal{F}$ denote the set of such test functions.

\begin{proposition} \label{QuantEquib2}
For archimedean places $v$, Theorem \ref{QuantEquib1} also holds for functions $f \in \cal{F}$. 
\end{proposition}

\begin{proof}
Let $M$ be an upper bound for $|z|$ for $z \in F$. By convolving with a smooth function $\chi_{\eps}$ supported on a disc of radius $\eps$, we obtain a sequence of smooth functions $f_{\eps}$ such that $f_{\eps} \to f$ uniformly on $\{|z| \leq M\}$. We may then apply Theorem \ref{QuantEquib1} to each such $f_{\eps}$. It then suffices to show that $\limsup \langle f_{\eps}, f_{\eps} \rangle \leq \langle f,f \rangle$. But as $\Lap f_{\eps}$ is equal to convolving $\Lap f$ over a disc of radius $\eps$, we obtain 
$$\left| \int g (\Lap f - \Lap f_{\eps}) \right| \leq \Lip(g) \eps$$
for any Lipschitz continuous function $g$. Then as $f$ is Lipschitz continuous with say constant $C$, $f_{\eps}$ is also Lipschitz continuous with constant $C$. Furthermore, we also have 
$$\int g \Lap f \leq (\sup g) \cdot |\Lap f|(\bb{P}^1(\bb{C}))$$
where $|\Lap f|$ is the trace measure. As $\Lap f$ is a finite signed measure, we know that $|\Lap f| \leq C'$ for some constant $C' > 0$. 
It follows that 
$$\langle f_{\eps}, f_{\eps} \rangle = \int f_{\eps} \Lap f_{\eps} \leq C \eps + \int f_{\eps} \Lap f \leq C \eps + C' \eps + \int f \Lap f =  (C+C') \eps + \langle \Lap f , \Lap f \rangle.$$
Taking $\eps \to 0$ gives us what we want. 
\end{proof}

\section{Quantitative Logarithmic Equidistribution for $\bb{G}_m$}
\subsection{Bounding distances to torsion points for $\bb{G}_m$}
The main tool that we will be needing is linear forms in logarithms. This was first developed by Baker (see \cite{Bak75}) that allowed one to give lower bounds to quantities of the form $|a_1^{b_1} \cdots a_n^{b_n} - 1|$ in terms of the heights of $a_i$ and $b_i$. 
\par 
The first result that we need is a theorem of Laurent, Mignotte and Nesterenko.

\begin{theorem} \label{Baker} (Theorem 3 \cite{LMN95})
Let $\alpha$ be an algebraic number with $|\alpha| = 1$ which is not a root of unity and let $b_1,b_2$ be positive integers. Define $\Lambda = b_1 i \pi - b_2 \log \alpha$. Let
$$D = [\bb{Q}(\alpha):\bb{Q}]/2, \quad a = \max \{20, 10.98|\log \alpha| + D h(\alpha) \},$$
$$H = \max\left \{17, \frac{\sqrt{D}}{10}, D \log(\frac{b_1}{2a} + \frac{b_2}{68.9}) + 2.35D + 5.03 \right\}.$$
Then
$$ \log |\Lambda| \geq -8.87aH^2.$$
\end{theorem}

\begin{corollary} \label{BakerCorollary1}
Let $\beta$ be an algebraic number with $|\beta| = 1$. Then for any $\epsilon > 0$, there exists a constant $C_{\epsilon} > 0$ such that for any root of unity $\zeta$ of order $n$, we have
$$\log |\beta - \zeta| \geq -C_{\epsilon} [\bb{Q}(\beta):\bb{Q}]^3 (h(\beta)+1) n^{\epsilon}.$$
\end{corollary}

\begin{proof}
Since $\beta$ and $\zeta$ are both on the unit circle, it suffices to bound how close their angles are. Thus we wish to bound the quantity $\log |\frac{k}{n} \pi i  - \log \beta|$. Since Theorem \ref{Baker} requires $b_1,b_2$ to be integers, we will instead bound $\log |k \pi i - n \log \beta|$ which incurs at most an extra $\log n$. As we may take $|\log \beta|$ to be less than $2 \pi$, and so we can take $a$ to be $ O([\bb{Q}(\beta):\bb{Q})] (h(\beta)+1))$. For $H$, as $b_1, b_2 = O(n)$, we can take $H$ to be $O(D \log n)$. Thus we can find a constant $C > 0$ such that
$$\log | k \pi i - n \log \beta| \geq -C [\bb{Q}(\beta):\bb{Q}]^3 (h(\beta)+1) (\log n)^2.$$
Since $\log n << n^{\epsilon}$ for all large enough $n$, we obtain that 
$$\log \left| \frac{k}{n} \pi i - \log \beta \right| \geq - C_{\epsilon} [\bb{Q}(\beta):\bb{Q}]^3 (h(\beta)+1) n^{\epsilon}$$
as desired.
\end{proof}

\begin{proposition} \label{PowerBound1}
Let $\vphi(z) = z^d$. For any $\epsilon > 0$, there exists a constant $C_{\epsilon}$ such that the following holds: Let $K$ be a number field and $\beta \in \bb{P}^1(K)$  be a non-preperiodic point. Let $v$ be an archimedean place of $K$. Then for any $\Gal(\ovl{\bb{Q}}/\bb{Q})$-orbit $|F|$ of preperiodic points, we have
$$\max_{x \in F} \log|x-\beta|^{-1}_v < C_{\epsilon}[K:\bb{Q}]^3(h(\beta)+1) |F|^{\epsilon}.$$
\end{proposition}

\begin{proof}
We first handle the case where $\beta$ does not lie on the unit circle. Then the distance to any root of unity can be bounded from below by $|\beta|_v - 1$. Now $|\beta|_v$ is a real algebraic number living in some field $K'$ of degree at most $2$ larger than $K$. Since $\beta \ovl{\beta} = |\beta|_v$ where $\ovl{\beta}$ is the complex conjugate, the height of $|\beta|_v$ is at most twice of that of $\beta$. By Proposition 5(b) of \cite{HS11}, we have $h(|\beta|_v - 1) \leq h(|\beta|_v) + \log 2$ and so $h(|\beta|_v - 1) \leq 2h(\beta) + \log 2$. Hence we conclude that
$$ \log | |\beta|_v - 1|^{-1} \leq [K':\bb{Q}](2h(\beta) + \log 2)$$ 
which gives us the constant $C$ we need.
\par 
Now let's us assume that $\beta$ lies on the unit circle and let $\zeta$ be a primitive $n^{th}$ root of unity. Applying Corollary \ref{BakerCorollary1}, as the Galois orbit $F$ of $\zeta$ has cardinality $\vphi(n) \geq \sqrt{n}$, we obtain a constant $C_{\epsilon} > 0$
$$\max_{x \in F} \log|x-\beta|^{-1}_v < C_{\epsilon}[K(\beta):K]^3 (h(\beta)+1)|F|^{\epsilon}$$
for any $\epsilon > 0$ as desired.
\end{proof}

\begin{proposition} \label{PowerBound2}
Fix a non-archimedean place $v$ of $\bb{Q}$ corresponding to the prime $p$ along with an extension to $\ovl{\bb{Q}}$. Let $D$ be a positive integer and let $\delta > 0$. Then there exists constants $C > 0$, depending on $p,\delta,D$, such that  for any $\beta \in \bb{P}^1(K)$ with $[K:\bb{Q}] < D$ and root of unity $\zeta$, we have
$$\log|\zeta-\beta|_v^{-1} < \delta$$
if $\deg_{\bb{Q}}(\zeta) > C$.
\end{proposition}

\begin{proof}
As each root of unity lies in a different residue class unless their order differs by a power of $p$, for $\beta \in K$ we have $\lambda_v(\beta,\zeta) = 0$ unless $\ord(\zeta) = p^k a_i$ for some finite set $\{a_1,\ldots,a_n\}$ of integers coprime to $p$. We can choose the $a_i$'s such that it works for any field extension $K$ of degree $D$, as the inertia degree of such a field is at most $D$.
\par 
Let $\zeta_m$ denote a primitive $m^{th}$ root of unity. Let $|\zeta_{p^k} - 1|_v = c$ . If $\ord(\zeta) = p^k a_i$, we may write $\zeta = \zeta_{p^k}^a \zeta_{a_i}^b$ for some natural numbers $a,b$ and thus conclude that $1 > |\zeta - \zeta_{a_i}^b| = c$. Hence 
$$c > |\beta - \zeta|_v \implies |\beta - \zeta_{a_i}^b|_v = c.$$
Due to ramification, we have $c = 1/p^{k-1}(p-1)$ and since we have only finitely many $\zeta_{a_i}$'s, the element $\beta - \zeta_{a_i}^j$ all live in some number field of fixed degree and thus it is impossible for our distance to be $c$ if $k$ is large enough. Hence if $k$ is large enough, $|\beta - \zeta|_v \geq c$ and so
$$\log |\beta - \zeta|_v^{-1} < \frac{1}{p^{k-1}(p-1)} \log p < \delta.$$
Hence if the order of $\zeta$ is larger than some constant $C$, we have $\log |\zeta - \beta|^{-1}_v < \delta$ as desired.
\end{proof}

As a simple corollary of the proof, we deduce the following too.

\begin{corollary} \label{PowerBound3}
Let $\vphi(z) = z^d$ and fix a non-archimedean place $v$ of $\bb{Q}$ corresponding to the prime $p$. Then for any $\beta \in \bb{P}^1(\ovl{\bb{Q}})$, there do not exist two distinct roots of unity $\zeta_1,\zeta_2$ such that
$$\log |\zeta_i - \beta|^{-1}_v > \frac{1}{p-1} \log p.$$
\end{corollary}

\begin{proof}
The condition implies that $|\zeta_i - \beta|_v < \frac{1}{p^{1/(p-1)}}$ and so $|\zeta_1 - \zeta_2|_v < \frac{1}{p^{1/(p-1)}}$. Thus $|1 - \zeta_1^{-1} \zeta_2|_v  < \frac{1}{p^{1/(p-1)}}$, which is impossible as $|1 - \zeta|_v$ is at least $1/p^{1/(p-1)}$ for any root of unity $\zeta$.
\end{proof}

\subsection{Proofs of Theorems \ref{IntroUniformIh1} and \ref{IntroUniformIh3}}

We now prove Theorems \ref{IntroUniformIh1} and \ref{IntroUniformIh3}. Fix a place $v \in S$.  We start with the function
$$\lambda_v(x,\beta) = \log^+|x|_v + \log^+|\beta|_v - \log |x-\beta|_v.$$
For any real number $M > 1$, we truncate our function to get a function
$$\lambda_{v,M}(x) = \log^+|x|_v + \log^+|\beta|_v + \min(\log M, -\log |x-\beta|_v).$$
We observe that if $|x-\beta| > M^{-1}$, then $\lambda_v(x,\beta) = \lambda_{v,M}(x)$. 

\begin{proposition} \label{Lipschitz1}
We have $\Lap \lambda_{v,M} = \delta_{\zeta(0,1)} - \delta_{\zeta(\beta,M^{-1})}$. Furthermore, $\lambda_{v,M}$ is Lipschitz continuous with Lipschitz constant $O(M)$.  
\end{proposition}

\begin{proof}
As $\Lap \min(\log M, - \log |x-\beta|_v) = \delta_{\infty} - \delta_{\zeta(\beta,M^{-1})}$, it follows that $\Lap \lambda_{v,M} = \delta_{\zeta(\beta,M^{-1})} - \delta_{\zeta(0,1)}$. To obtain a Lipschitz constant for the Euclidean metric, observe that $\log^+|x|_v$ has Lipschitz constant $1$ and that $\log|x-\beta|_v$ has a Lipschitz constant of $\frac{1}{M}$ for $|x-\beta| > M^{-1}$. Hence $M+1$ serves as a Lipschitz constant for our truncated function.
\end{proof}

We will also need the following lower bound on size of the $\Gal(\ovl{K}/K)$-orbit of a root of unity. 

\begin{proposition} \label{Galois1}
Let $\zeta$ be a $n^{th}$ root of unity. Then $[K(\zeta):K] \geq \frac{\sqrt{n}}{[K:\bb{Q}]}$ for all sufficiently large $n$.
\end{proposition}

\begin{proof}
We know that $[\bb{Q}(\zeta):\bb{Q}] = \phi(n)$ where $\phi(n)$ is the number of positive integers $\leq n$ that are coprime to $n$. It is known that for all large enough $n$, we have $\phi(n) \geq \sqrt{n}$ (in fact we may replace $\sqrt{n}$ with $n^{1-\eps}$). Thus $[\bb{Q}(\zeta):\bb{Q}] \geq \sqrt{n}$. It now suffices to note that 
$$[K(\zeta):K] \cdot [K:\bb{Q}] = [K(\zeta):\bb{Q}] \geq [\bb{Q}(\zeta):\bb{Q}] \geq \sqrt{n}$$       
as desired.
\end{proof}

We now begin the proof of Theorems \ref{IntroUniformIh1} and \ref{IntroUniformIh3}. First, let $F$ be a $\Gal(\ovl{K}/K)$-invariant set of roots of unity and $\beta \in L$ for some extension $L$ of $K$ with degree $D$. Now fix a place $v$ of $L$ and let's say 
$$\max_{x \in F} \log |x-\beta|^{-1}_v \leq C$$
for some constant $C > 0$. Fix some real $M > 0$ and we split $F$ into two disjoint sets $F_1,F_2$, where $\log |x-\beta|^{-1}_v \leq M$ for $x \in F_1$ and $F_2 = F \setminus F_1$. Then for $x \in F_1$, we have $\lambda_v(x,\beta) = \lambda_{v,M}(x)$. In particular, we have
$$\left|\frac{1}{|F|}\sum_{x \in F} \lambda_{v,M}(x) - \frac{1}{|F|} \sum_{x \in F} \lambda_v(x,\beta) \right| \leq C \frac{|F_2|}{|F|}.$$
Now applying Proposition \ref{QuantEquib2} along with Proposition \ref{Lipschitz1}, we obtain that
$$\left| \frac{1}{|F|} \sum_{x \in F} \lambda_{v,M}(x) - \int \lambda_{v,M}(x) d \mu_{v} \right| \leq \frac{M}{|F|^{\delta/\kappa}} + \left(\frac{ c M \log |F|}{|F|} \right)^{1/2}.$$
As $\mu_v(D(\beta,x)) = O(x)$, We may then bound 
$$\left| \int \lambda_{v,M}(x) d \mu_v - \int \lambda_v(x,\beta) \right| \leq O\left(-\int_0^{\frac{1}{M}} \log |x| dx \right) \leq O\left( \frac{\log M}{M} \right).$$
We thus obtain the following Proposition.

\begin{proposition} \label{QuantLogBound1}
Let $v$ be a place of $K$ that is extended to $\ovl{K}$. Let $F$ be the $\Gal(\ovl{K}/K)$-orbit of some root of unity $\zeta$. Let $\max_{x \in F} \log|x-\beta|^{-1}_v = C.$ 
Then we have 
$$\left|\frac{1}{|F|} \sum_{x \in F} \lambda_v (x,\beta) - \int \lambda_v(x,\beta) d\mu_v \right| \leq \frac{C}{|F|} + O_{\eps} \left([K:\bb{Q}]|F|^{-1/2 + \eps} \right).$$
\end{proposition}

\begin{proof}
Let $\zeta$ be a $n^{th}$ root of unity and let $D = [K:\bb{Q}]$. By Proposition \ref{Galois1}, we know that $|F| \geq \frac{n^{1/2}}{D}$. In particular if $v$ is an archimedean place and $|F| \geq D^4$ and we take $M \geq |F|^4$, then there is at most one $\zeta$ inside $F$ for which $\log |x-\beta|^{-1}_v \geq \log M$. Hence we may take $|F_2| \leq 1$ for archimedean $v$ and by Proposition \ref{PowerBound3}, we may do the same for non-archimedean $v$. Now taking $\delta$ sufficiently large such that $\delta/\kappa \geq 6$, we obtain 

$$\left|\frac{1}{|F|} \sum_{x \in F} \lambda_v (x,\beta) - \int \lambda_v(x,\beta) d\mu_v \right| \leq \frac{C}{|F|} + \frac{|F|^4}{|F|^6} + O\left([K:\bb{Q}]\frac{\log |F|}{|F|^{1/2 - \eps}} \right).$$
This reduces to 
$$\left|\frac{1}{|F|} \sum_{x \in F} \lambda_v (x,\beta) - \int \lambda_v(x,\beta) d\mu_v \right| \leq \frac{C}{|F|} + O_{\eps} \left([K:\bb{Q}]|F|^{-1/2 + \eps} \right)$$
as desired. 
\end{proof}

\begin{proof}[Proof of Theorem \ref{IntroUniformIh1}] First, let $S'$ be the set of places above $S$. Observe that $|S'| \leq [K:\bb{Q}]|S| \leq D |S|$. By Proposition \ref{PowerBound2}, there exists some $C'$ depending on $S$ and $[K:\bb{Q}]$ such that if $\deg_{\bb{Q}}(\zeta) > C'$, then 
$$\log |\zeta-\beta|_v^{-1} \leq 1$$
for any non-archimedean place $v \in S$. Now let $\beta \in K^{\times} \setminus \mu_{\infty}$ and assume that $\zeta \in \mu_{\infty}$ is $S'$-integral relative to $\beta$. Let $F$ be the $\Gal(\ovl{\bb{Q}}/K)$-orbit of $\zeta$. Then we have $\lambda_v(x,\beta) = 0$ for all $v \not \in S'$ and $x \in F$. Hence
\begin{equation} \label{eq: Height1}  h(\beta) = \frac{1}{|F|}\sum_{v \in M_K}  \sum_{x \in F} N_v\lambda_v (x,\beta) = \frac{1}{|F|} \sum_{v \in S'} \sum_{x \in F} N_v \lambda_v(x,\beta).
\end{equation}

By Proposition \ref{QuantLogBound1}, as for non-archimedean $v$ we have $\log |x-\beta|_v^{-1} \leq 1$ for all $x \in F$ and for archimedean $v$, by Proposition \ref{PowerBound1} we have 
$$\log |x-\beta|_v^{-1} \leq O_{\eps}(D^3 (h(\beta)+1) |F|^{\eps}),$$
we obtain
\begin{equation} \label{eq: QuantLog1}
\left| \frac{1}{|F|} \sum_{x \in F} \lambda_v(x,\beta) - \int \lambda_v(x,\beta) d \mu_v \right| 
\end{equation}
$$\leq O_{\eps}(D^4 (h(\beta)+1) |F|^{-1 + \eps}) +  O_{\eps}(D |F|^{-1/2 + \eps}).
$$
Hence if $|F|$ is large enough depending only on $D$ and $S$, we obtain 
$$\left|\frac{1}{|F|} \sum_{v \in S} \sum_{x \in F} N_v \lambda_v(x,\beta) - \int N_v\lambda_v(x,\beta) d \mu_v \right| \leq \frac{h(\beta)+1}{D^5}.$$
By Dobrolowski \cite{Dob79}, we know that $h(\beta) \geq O(\frac{1}{D^{3/2}})$. Since $\int \lambda_v(x,\beta) d \mu_v = 0$, we have 
$$\frac{1}{|F|} \sum_{v \in S} \sum_{x \in F} N_v \lambda_v(x,\beta) < h(\beta)$$
which is a contradiction to \eqref{eq: Height1}. 
\end{proof}

To prove Theorem \ref{IntroUniformIh3}, we have to be more careful with our estimates. 

\begin{proof}[Proof of Theorem \ref{IntroUniformIh3}]
First, using Corollary \ref{PowerBound3}, we know that for each non-archimedean place $v$ of $K$, there is at most one root of unity such that $\log |\zeta - \beta|_v^{-1} \geq 1$. Hence up to having $|S_{\fin}|$ exceptions, we may assume that $\log |\zeta - \beta|_v^{-1} \leq 1$ for all non-archimedean places $v$. By Dobrolowski's bound \cite{Dob79}, we know there exists $c$ such that $h(\beta) > \frac{1}{cD^{2}-1}.$
Using \eqref{eq: QuantLog1}, we again get
$$\left|\frac{1}{|F|} \sum_{x \in F} \lambda_v(x,\beta) - \int \lambda_v(x,\beta) d \mu_v \right|$$
$$\leq O_{\eps}(D^4 (h(\beta)+1)|F|^{-1 + \eps}) + O_{\eps}(D |F|^{-1/2 + \eps}).$$
Summing up over $s \in S$, we obtain
$$\left|\frac{1}{|F|} \sum_{v \in S} \sum_{x \in F} N_v \lambda_v(x,\beta) - \sum_{v \in S} N_v \int \lambda_v(x,\beta) d \mu_v \right|$$ 
$$\leq O_{\eps}(|S| D^4 (h(\beta)+1)|F|^{-1 + \eps}) + O_{\eps}(|S| D |F|^{-1/2 + \eps}).$$
Hence if $|F| > O(c^3 |S|^{3} D^{6})$, we obtain that 
$$\left|\frac{1}{|F|} \sum_{x \in F} \sum_{v \in S} N_v \lambda_v(x,\beta) \right| \leq \frac{h(\beta)+1}{cD^{2}}.$$
Then \eqref{eq: Height1} tells us that 
$$h(\beta) \leq \frac{h(\beta)+1}{c D^2} \implies h(\beta) \leq \frac{1}{cD^2-1}$$
which is a contradiction as desired.
\end{proof}

\section{Quantitative Equdistribution for Elliptic Curves}
We now move onto the case of elliptic curves $E$. Over number fields, quantitative versions of equidistribution has been proven by Baker--Petsche \cite{BP05} and in the function field case, by Petsche \cite{Pet09}. Petsche's result is explicit and in a similar form to Favre--Rivera-Letelier's theorem but the results in Baker--Petsche are not in the same form and it does not seem straightforward to deduce an explicit bound for a given test function $f$. 
\par 
We will first show that when $v$ is a non-archimedean place of a number field $K$, it is possible to combine the results of \cite{BP05} and \cite{Pet09} to obtain a statement similar to Theorem \ref{QuantEquib1}. When $v$ is an archimedean place, we will establish an analogous result by using Favre--Rivera-Letelier's approach instead.
\par 
Let $E$ be an elliptic curve over a number field $K$ with semistable reduction and let $L$ be a finite extension of $K$. We will be interested in establishing bounds on the equidistribution of $\Gal(\ovl{K}/L)$-orbits $F$ of torsion points in terms of $|F|$ and $[L:K]$, assuming that the data of our elliptic curve $E$ remains fixed. 
\par 
We first state some preliminaries on elliptic curves that we require. For an elliptic curve $E$ over a number field $K$, we may define a Neron--Tate height $h_E: E(\ovl{K}) \to \bb{R}_{\geq 0}$ such that $h_E([n]x) = n^2 h_E(x)$ where $[n]$ is the multiplication by $n$ map. 
\par 
Similar to the Weil height $h(x)$, we have a local decomposition 
$$h_E(x) = \sum_{v \in M_K} N_v \lambda_v(x)$$
where $\lambda_v: E(\bb{C}_v) \setminus \{O\} \to \bb{R}$ is a function that has a singularity at the origin $O$. One may consult \cite[Chapter VI]{Sil94} for explicit definitions of $\lambda_v$. For an archimedean place $v$, we will let $\mu_v$ denote the Haar measure on $E(\bb{C}_v)$. For non-archimedean places, we will let it denote the canonical probability measure supported on the skeleton $\Sigma$ of the Berkovich analytification $E^{\an}_v$ at the place $v$ as defined in \cite[Section 5.2]{BP05}. For each place $v$, we have that $\int \lambda_v(x) d \mu_v = 0$. For more background on $E^{\an}_v$, the reader may consult \cite[Section 3]{BP05}. 
\par 
Before we begin, we need a preliminary lemma on the size of Galois orbits of torsion points.

\begin{proposition} \label{EllipticGalois1}
Let $E$ be an elliptic curve over a number field $K$ and $L$ a finite extension of $K$. Let $z \in E(\ovl{K})$ be a $n$-torsion point. Then if $F$ denotes the $\Gal(\ovl{L}/L)$-orbit of $z$, we have 
$$|F| \geq c_{\eps} \frac{n^{1-\eps}}{[L:K]} $$
where the constant $c_{\eps} > 0$ only depends on the elliptic curve $E,K$ and $\eps > 0$. 
\end{proposition}

\begin{proof}
As in Proposition \ref{Galois1}, it suffices to show that the $\Gal(\ovl{K}/K)$-orbit of $z$ has orbit size $\geq c n^{1-\eps}$. But this follows immediately from (45) and (46) of \cite{BIR08}.     
\end{proof}

\subsection{The non-archimedean case.} Recall that our elliptic curve $E$ is defined over $K$ and that $L$ is a finite extension of $K$. Let $v$ be a non-archimedean place of $L$. Given a set of points $Z = \{P_1,\ldots,P_n\}$ in $E(\bb{C}_v)$, we define the local discrepancy as 
$$\Lambda_v(Z) = \frac{1}{n^2} \sum_{1 \leq i \not = j \leq n} \lambda_v(P_i - P_j).$$
Baker--Petsche then defines a smoothened analogue $D_v(Z)$ such that $D_v(Z) \geq 0$ for all places $v$ and 
$$\Lambda_v(Z) = D_v(Z) - \frac{1}{12n} \log^+|j_E|_v$$
where $j_E$ is the $j$-invariant of our elliptic curve $E$. By \cite[Theorem 8]{BP05}, as each $D_v(z)$ is non-negative, we have the following upper bound 
\begin{equation} \label{eq: LocalDiscrepancy1}
D_v(Z) \leq [L:\bb{Q}] \left( 4 h(Z) + \frac{1}{n} \left(\frac{1}{2} \log n + \frac{1}{12} h(j_E) + \frac{16}{5} \right) \right).
\end{equation}
Now for a real number $M$ with $\log M > \frac{1}{12} \log^+|j_E|_v$ and $\beta \in E(\bb{C}_v)$, consider the function $\lambda_{v,M}(x) = \min(\log M, \lambda_v(x-\beta))$. We will now apply Theorem 10 of \cite{Pet09} to obtain an estimate on
$$\left|\frac{1}{|F|} \sum_{x \in F} \lambda_{v,M}(x) - \int \lambda_{v,M} d \mu_v \right|.$$

\begin{proposition} \label{QuantNonArchElliptic1}
Let $F$ be a $\Gal(\ovl{K}/L)$-invariant finite subset of $E(\ovl{K})$. We have 
$$\frac{1}{[L:\bb{Q}]}\left|\frac{1}{|F|} \sum_{x \in F} \lambda_{v,M}(x) - \int \lambda_{v,M} d \mu_v \right| \leq \sqrt{\log M} \left( O\left(h(Z) + \frac{\log |F|}{|F|} \right) + \frac{\log M}{|F|} \right)^{1/2}$$
\end{proposition}

\begin{proof}
First, by \cite[Section 3.2]{BP05}, we may write $\lambda_v(x-\beta)$ as $i_v(x,\beta) + j_v(x,\beta)$ where $i_v(x,\beta)$ is a local intersection term and $j_v(x,\beta)$ factors through the retraction map $r: E^{\an}_v \to \Sigma$. Since $j_v(x,\beta) \leq \frac{1}{12} \log^+ |j_E|_v$, if $\lambda_{v,M}(x) \geq M$ then it must be that $x$ and $\beta$ lie in the same residue disc. Hence $\lambda_v(x-\beta) = i_v(x,\beta) + \frac{1}{12} \log^+|j_E|_v$. 
\par 
We may identify this residue disc with the Berkovich unit disc $\bb{D}_{\an}(0,1)$ with $\beta$ being $0$. On the skeleton, we may bound $\int |\lambda_{v,M}'|^2 d \mu_v$ by a constant independent of $M$. Then on the path $\Gamma'$ from the Gauss point $\zeta(0,1)$ to $\zeta(\beta,M)$, our function $\lambda_{v,M}(x)$ is increasing with derivative $1$ with respect to the path metric and is locally constant on $\bb{D}_{\an}(0,1) \setminus \Gamma'$. Thus $\lambda_{v,M}(x)$ belongs to the space $S_{\Gamma}(E,\bb{R})$ where $\Gamma = \Gamma' \cup \Sigma$ with $|\lambda_{v,M}'| = 1$ on $\Gamma'$ and $0$ otherwise, and $l_0(\Gamma) = \log M$. 
\par 
Applying \cite[Theorem 10]{Pet09}, we obtain 
$$\left|\frac{1}{|F|} \sum_{x \in F} \lambda_{v,M}(x) - \int \lambda_{v,M} d \mu_v \right| \leq \sqrt{\log M} \left( D_v(F) + \frac{\log M}{|F|} \right)^{1/2}.$$
We now use \eqref{eq: LocalDiscrepancy1} to upper bound $D_v(F)$ and obtain 
$$\frac{1}{[L:\bb{Q}]} \left|\frac{1}{|F|} \sum_{x \in F} \lambda_{v,M}(x) - \int \lambda_{v,M} d \mu_v \right| \leq \sqrt{\log M} \left(O\left(h(Z) + \frac{\log |F|}{|F|} \right) + \frac{\log M}{|F|} \right)^{1/2}$$
as desired.
\end{proof}

We now specialize by taking $M = |F|$ and assuming $Z$ are all torsion points, so that $h_E(Z) = 0$.  Also, as $d\mu_v$ is supported on the skeleton, we have $\int \lambda_{v,M}(x) d \mu_v = \int \lambda_v(x) d \mu_v = 0$. If we let $C = \max_{x \in F} \lambda_v(x - \beta)$ and let $N$ be the number of elements of $|F|$ for which $\lambda_v(x-\beta) > N$, we obtain 
\begin{equation} \label{eq: NonArchQuantLog1}
\left|\frac{1}{|F|} \sum_{x \in F} \lambda_v(x-\beta) \right| \leq [L:\bb{Q}] O\left(\frac{\log |F|}{|F|} \right)^{1/2} + \frac{CN}{|F|}.
\end{equation}

\subsection{The archimedean case.} We now assume that $v$ is an archimedean place. Following Favre--Rivera-Letelier, we will introduce an energy pairing between two measures $\mu, \mu'$ on $E(\bb{C}_v)$. We refer readers to \cite[Sections 2.1 and 2.2]{BP05} for background on the Laplacian $\Lap$ for $E(\bb{C})$. In summary, one may take an isomorphism $E(\bb{C}) \simeq \bb{C}/\Lambda$ for a normalized lattice $\Lambda = \bb{Z} + \tau \bb{Z}$ where $\tau = a+bi$ with $b > 0$. For a continuous function $g$ on $E(\bb{C})$, we define $\Lap g$ to be 
$$\Lap g = \frac{b}{2 \pi} \left(\frac{\partial^2}{\partial x^2} + \frac{\partial^2}{\partial y^2} \right) g.$$
The local Neron function $\lambda_v$ satisfies $\Lap \lambda_v = \mu - \delta_O$ where $O$ is the origin. 
\par 
To define the energy pairing, we require a replacement for $\log |x-y|$. The main property of $\log |x-y|$ is that its Laplacian is $\delta_{y} - \delta_{\infty}$. Using the origin $O$ as our replacement for $\infty$, we wish to find a function $g_v(x,y)$ on $E(\bb{C})$ whose Laplacian with respect to $x$ is $\delta_{y} - \delta_{O}$. A natural candidate is given by
$$g_v(x,y) = -\lambda_v(x-y) + \lambda_v(x) + \lambda_v(y).$$
We have $\Lap g_v(x,y) = (\delta_y - \mu) + (\mu - \delta_O) = \delta_y - \delta_O$. Given a measure $\mu$ which has locally continuous potentials, we may define a Green's function $f_v(x) = \int g_v(x,y) d \mu(y)$ and observe that 
$$\Lap f_v(x) = \Lap \int g_v(x,y) d \mu(y) = \int \Lap_x g_v(x,y) d \mu(y) = \int (\delta_y - \delta_O) d \mu(y) = \mu - \mu(E(\bb{C})) \delta_O.$$
Now given two measures $\mu,\mu'$, which are a union of finite delta masses or measures with locally continuous potential, we define its local energy as 
$$(\mu,\mu') = -\iint_{E x E \setminus \Diag} g_v(x,y) d \mu(x) d \mu'(y).$$
Analogous to \cite[Proposition 2.6]{FRL06}, we deduce that if $\mu$ is a signed measure with $\mu(E(\bb{C})) = 0$ and if $\mu = \mu_1 - \mu_2$ such that $\mu_1,\mu_2$ have locally continuous potentials, then $\mu = \Lap f$ for some continuous function on $E(\bb{C})$. Furthermore for such measures, we have $(\mu,\mu) \geq 0$ with equality happening iff $\mu = 0$. 
\par 
Now, if $F$ is a finite set of points on $E(\bb{C})$, we will let $[F] = \frac{1}{|F|} \sum_{x \in F} \delta_x$. Similar to the case on $\bb{P}^1$, under the isomorphism $E(\bb{C}) \simeq \bb{C}/\Lambda$, we will let $\zeta(x,\eps)$ denote a circle of radius $\eps$ around $x$ and we let $\delta_{\zeta(x,\eps)}$ be the uniform probability measure on this circle. We next compute the self-pairing $([F] -\mu_v,[F] - \mu_v)$. 

\begin{proposition} \label{EnergyPairing2}
We have 
$$([F] - \mu_v ,[F] - \mu_v) = \frac{1}{|F|^2}\sum_{x \not = x'} \lambda_v(x - x').$$
\end{proposition}

\begin{proof}
First, we have 
$$([F],[F]) = \frac{1}{|F|^2} \sum_{x \not = x'} \lambda_v(x-x') - \frac{2}{|F|} \sum_{x \in F} \lambda_v(x).$$
Next, as $\int \lambda_v(x) d \mu_v = 0$, it follows that $\int \lambda_v(x-y) d \mu_v = 0$ too as $\mu_v$ is invariant under translation. Hence $(\mu_v, \mu_v) = 0$. Finally, we have
$$-([F],\mu_v) = -\frac{1}{|F|} \sum_{x \in F} \int \left( \lambda_v(x - y) - \lambda_v(x) - \lambda_v(y) \right) d \mu_v = \frac{1}{|F|} \sum_{x \in F} \lambda_v(x).$$

Thus summing up, we get
$$([F]-\mu_v, [F]-\mu_v) = \frac{1}{|F|^2} \sum_{x \not = x'} \lambda_v(x-x')$$
as desired.
\end{proof}

Given an $\eps > 0$, we regularize $[F]$ by replacing each delta mass with $\delta_{\zeta(x,\eps)}$, the uniform probability measure on the circle of radius $\eps$ from $x$. As in the case on $\bb{P}^1$, we denote this measure by $[F]_{\eps}$. We now want to estimate 
$$([F]-\mu_v,[F]-\mu_v) - ([F]_{\eps} -\mu_v, [F]_{\eps} - \mu_v).$$
To do so, we need to obtain some bounds regarding $\lambda_v(x)$. 

\begin{proposition} \label{LocalHeight1}
Under the isomorphism $E(\bb{C}_v) \simeq \bb{C}/\Lambda$, the function $\lambda_v(z)$ viewed as a periodic function on $\bb{C}$ is smooth outside the lattice $\Lambda$. On a small disc around the origin, the function $\lambda_v(z) + \log |z|$ is smooth and extends to the origin.  
\end{proposition}

\begin{proof}
By \cite[Chapter VI, Proposition 3.1]{Sil94}, we know that $\lambda_v(z)$ is real-analytic away from the origin. Near the origin, we may write it as $f_1(z) + \log |f_2(z)|$ where $f_1(z)$ is a smooth function and $f_2$ is a meromorphic function with a pole at the origin. Thus $\log |f_2(z)| + \log |z|$ extends to a smooth function at the origin as desired.     
\end{proof}

We then get the following two bounds. 

\begin{proposition} \label{ArchRegBound1}
For $z \in \bb{C}$ with $|z|$ sufficiently small, we have
$$\left| \lambda_v(z) + \log |z| \right| \leq O(1).$$
\end{proposition}

\begin{proof}
This follows from the fact that $\lambda_v(z) + \log |z|$ is a smooth function and hence a bounded one near the origin. 
\end{proof}

\begin{proposition} \label{ArchRegBound2}
Given $z \in \bb{C}/\Lambda$, let $d(z)$ be the distance from $z$ to the origin. Then for all $y \in \bb{C}/\Lambda$ satisfying $|y-z| < \frac{1}{2} d(z)$, we have
$$|\lambda_v(y) - \lambda_v(z)| \leq \frac{1}{d(z)}O(|y-z|).$$
\end{proposition}

\begin{proof}
Let $\eps > 0$ be a constant such that if $d(z) > \eps$, then $d(y) > \frac{\eps}{2}$ for any $y$ such that $|y-z| \leq \frac{1}{2} d(z)$. We may choose $\eps$ to be $\frac{1}{2}$ of the distance between the four vertices of $\Lambda$. 
\par 
Now outside of $D(O,\frac{\eps}{2}) \subseteq E(\bb{C})$, our function $\lambda_v(z)$ is smooth and thus Lipschitz continuous with some constant $C$. Thus if $d(z) > \eps$ and $|y-z| < \frac{1}{2} d(z)$, then $y \not \in D(0,\frac{\eps}{2})$ and hence $|\lambda_v(y) - \lambda_v(z)| \leq C |y-z|$. 
\par 
Now let's say $z \in D(O,\eps)$. Then $y \in D(O,2 \eps)$ and we know that $\lambda_v(y) + \log |z|$ is a smooth function on $D(O, 2 \eps)$. Thus we have 
$$\lambda_v(y) - \lambda_v(z) + \log |y| - \log |z| \leq O(|y-z|).$$
On the other hand, we may bound $\log |y| - \log |z|$ by 
$$\log |y| - \log |z| = \log |\frac{y}{z}| = \log |1 + \frac{y-z}{z}| \leq 2 \left|\frac{y-z}{z} \right| = \frac{1}{d(z)} O(|y-z|).$$
Here, the first inequality follows from the fact that $|\frac{y-z}{z}| \leq \frac{1}{2}$ and using the Taylor expansion of $\log(1+z)$. Putting the two upper bounds together give us our Proposition.
\end{proof}

Given a finite set $F \subseteq E(\bb{C}_v)$ of points, we now set
$$d_v(F) = \min_{z \not = z' \in F} \{ |z-z'|_v, d(z)\}.$$
We now bound the energy difference that occurs when regularizing our discrete set $F$. For a point $z \in E(\bb{C})$, for $\eps$ small enough the notion of a disc of radius $\eps$ makes sense via the isomorphism $E(\bb{C}) \simeq \bb{C}/\Lambda$. We let $\delta_{z,\eps}$ denote the uniform probability measure on the boundary of the disc centered at $z$ and with radius $\eps$. Given $[F] = \frac{1}{|F|} \sum_{z \in F} \delta_z$, we let $[F]_{\eps} = \frac{1}{|F|} \sum_{z \in F} \delta_{z,\eps}$. 

\begin{proposition} \label{ArchRegBound3}
For $\eps \leq \frac{1}{4} d(F)$, we have
$$([F]_{\eps} - \mu_{v} , [F]_{\eps} - \mu_{v}) \leq ([F] - \mu_{v}, [F] - \mu_{v}) + O\left(\frac{1}{d_v(F)} \left( \eps + \frac{\log \eps^{-1}}{|F|} \right) \right).$$
\end{proposition}

\begin{proof}
Note that if $\mu$ has locally continuous potentials, then
$$(\mu- \mu_v, \mu- \mu_v) = \iint \lambda_v(x-y) d \mu(x) d \mu(y).$$
Hence we may write
$$([F]-\mu_v, [F]-\mu_v) - ([F]_{\eps} - \mu_v, [F]_{\eps} - \mu_v) $$
$$= \frac{1}{|F|^2} \sum_{z \not = z' \in F} \left(\lambda_v(z-z') - \int \lambda_v(x-y) \delta_{z,\eps}(x) \delta_{z',\eps}(y) \right) + \frac{1}{|F|^2} \sum_{z \in F} \iint \lambda_v(x-y) \delta_{z,\eps}(x) \delta_{z,\eps}(y).$$
We first bound $\lambda_v(z-z') - \int \lambda_v(x-y) \delta_{z,\eps}(x) \delta_{z',\eps}(y)$. For $x \in \supp \delta_{z,\eps}$ and $y \in \delta_{z',\eps}$, we have 
$$|(x-y) - (z-z')| \leq 2 \eps \leq \frac{1}{2} d(z-z').$$
We may then apply Proposition \ref{ArchRegBound2} to obtain 
$$|\lambda_v(x-y) - \lambda_v(z-z')| \leq \frac{1}{d(z-z')} O(\eps).$$
We next bound $\iint \lambda_v(x-y) \delta_{z,\eps}(x) \delta_{z,\eps}(y)$. As $|x-y|$ is sufficiently small, we know that $\lambda_v(x-y) + \log |x-y|$ is a bounded function. Thus we get
$$\iint \lambda_v(x-y) \delta_{z,\eps}(x) \delta_{z,\eps}(y) \leq -\iint \log |x-y| \delta_{z,\eps}(x) \delta_{z,\eps}(y) + O(1) = \log \eps^{-1} + O(1).$$
Putting it together, we obtain 
$$([F]-\mu_v, [F]-\mu_v) - ([F]_{\eps} - \mu_v, [F]_{\eps} - \mu_v)$$
$$\leq \frac{1}{|F|^2} \sum_{z \not = z' \in F} \frac{1}{d(z-z')} O(\eps) + \frac{1}{|F|^2} \sum_{z \in F} \left(\log \eps^{-1} + O(1) \right) \leq  O\left(\frac{1}{d(F)} \left( \eps + \frac{\log \eps^{-1}}{|F|} \right) \right)$$
as desired. 
\end{proof}

We now have all the ingredients needed to prove our quantitative equidistribution theorem for archimedean places. 

\begin{proposition} \label{ArchRegBound4}
Let $F = \{z_1,\ldots,z_n\}$ be a $\Gal(\ovl{K}/L)$-invariant susbet of $E(\ovl{K})$. Let $\eps > 0$ be a constant such that $\eps < \frac{1}{4 d_v(F)}$ for every archimedean place $v$. Then
$$\frac{1}{[L:\bb{Q}]}([F]_{\eps} - \mu_v, [F]_{\eps} - \mu_v) \leq  h_E(F) + \frac{1}{|F|} O\left( \log \eps^{-1} \right) + \eps \sum_{v \text{ arch. }} \frac{1}{d_v(F)}$$
\end{proposition}

\begin{proof}
Let $F = \{z_1,\ldots,z_n\}$ be a $\Gal(\ovl{K}/L)$-invariant subset of $E(\ovl{K})$. First, we know that for non-archimedean $v$, we have 
$$\frac{1}{|F|^2}\sum_{z \not = z' \in F} \lambda_v(z - z') \geq -\frac{1}{12|F|} \log^+ |j_E|_v.$$
We also have the upper bound
$$\frac{1}{|F|^2} \sum_{v \in M_K} \sum_{z \not = z' \in F} N_v \lambda_v(z-z') \leq  \frac{4}{|F|} \sum_{z \in F} h_E(z).$$
Thus we have 
$$\sum_{v \text{ arch. }} N_v([F]-\mu_v, [F]-\mu_v)_v = \frac{1}{|F|^2} \sum_{v \text{ arch. }} \sum_{z \not = z' \in F} N_v \lambda_v(z-z') \leq \frac{4}{|F|}\sum_{z \in F} h_E(z) +  O\left(\frac{1}{12|F|} \right).$$
Applying Proposition \ref{ArchRegBound3}, we obtain
$$\sum_{v \text{ arch. }} N_v ([F]_{\eps} - \mu_v, [F]_{\eps} - \mu_v)_v \leq \frac{1}{|F|} O\left(\sum_{z \in F} h_E(z) + [L:\bb{Q}]\log \eps^{-1} \right) + \eps \sum_{v \text{ arch. }} \frac{1}{d_v(F)}. $$
Here, we note that the constants in Proposition \ref{ArchRegBound3} may be made independent of $L$ as for any archimedean place $v$ of $L$, if it lies above the place $v'$ of $K$ then the constants are the same as for $v'$. Since each $([F]_{\eps} - \mu_v, [F]_{\eps}-\mu_v)$ is non-negative, we obtain an upper bound on each of them as desired.    
\end{proof}

Finally if $f$ is a continuous function on $E(\bb{C}_v)$, we have 
$$\left|\int f d([F]_{\eps} - \mu_v) \right| \leq (\Lap f, \Lap f)^{1/2} ([F]_{\eps} - \mu_v, [F]_{\eps}-\mu_v)^{1/2}$$
and
$$\left|\int f d([F]) - \int f d([F]_{\eps}) \right| \leq \eps \Lip(f).$$
If we assume that $F$ is a set of $m$-torsion points, so that $\hat{h}(F) = 0$, then we may also bound $d_v(F)$ from below by $\frac{c}{m}$ for some $c > 0$. Then since we fix $K$ and our elliptic curve $E$, by Proposition \ref{EllipticGalois1} we know that $m = O([L:\bb{Q}] n^2)$ and so taking $\eps = |F|^{-6}$ gives us the following bound. 

\begin{proposition} \label{ArchQuant1}
Let $F$ be a $\Gal(\ovl{K}/L)$-invariant set of torsion points with $|F| > 1$ and $f: E(\bb{C}_v) \to \bb{R}$ a continuous function for an archimedean place $v$ of $L$. Then
$$\frac{1}{[L:\bb{Q}]}\left|\frac{1}{|F|}\sum_{x \in F} f(x) - \int f(x) d \mu_v \right| \leq \frac{\Lip(f)}{|F|^6} + (\Lap f, \Lap f)^{1/2} O\left(\frac{\log |F|}{|F|} \right)^{1/2} + \frac{[L:\bb{Q}]^2}{|F|^4}.$$
\end{proposition}

Now again for fixed $M > 1$, we will define a truncated version of $\lambda_v(x,\beta)$. We let $\delta_{\zeta(0,M^{-1})}$ be probability measure on the circle centered at $0$ with radius $M^{-1}$. We then define
$$\lambda_{v,M}(x) = -\int \left(\lambda_v(x-y) - \lambda_v(x) \right) ( \mu_v - \delta_{\zeta(0,M^{-1})}) = \int \lambda_v(x-y) (\delta_{\zeta(0,M^{-1})})$$
as $\int \lambda_v(x-y) d \mu_v = 0$ and $\int \lambda_v(x) (\delta_{\zeta(0,M^{-1})} - \mu_v) = 0$. 

\begin{proposition} \label{ArchQuant2}
Assume that $M$ is sufficiently large depending on $E$. For s$x \in E(\bb{C}_v)$ such that $d(x) > M^{-1/2}$, we have $|\lambda_{v,M}(x) - \lambda_v(x)| \leq O(M^{-1/2})$. We also have the upper bounds 
$$\Lip(\lambda_{v,M}) \leq O(M), \sup |\lambda_{v,M}| \leq O(\log M) \text{ and }\langle \Lap \lambda_{v,M} , \Lap \lambda_{v,M}  \rangle \leq O(\log M).$$
\end{proposition}

\begin{proof}
If $x$ satisfy $d(x) > M^{-1/2}$, we may then apply Proposition \ref{ArchRegBound2} to get $|\lambda_v(x-y) -\lambda_v(x)| \leq O(M^{-1/2})$ for $y \in \supp \delta_{\zeta(0,M^{-1})}$. Hence 
$$\left| \lambda_v(x) - \int \lambda_v(x-y) (\delta_{\zeta(0,M^{-1})}) \right| \leq O(M^{-1/2})$$
as desired. Clearly $\Lap \lambda_{v,M} = \delta_{\zeta(0,M^{-1})} - \mu_v$ and hence we may bound $(\Lap \lambda_{v,M}, \Lap \lambda_{v,M}) \leq 2 \sup |\lambda_{v,M}|$. To bound $\sup |\lambda_{v,M}|$, we have to bound $\int \lambda_v(x-y) \delta_{\zeta(0,M^{-1})}$. As $M$ may be taken sufficiently large, we split into two cases. The first is when $x$ is not near the origin. Then for $y \in \supp \delta_{\zeta(0,M^{-1})}$, $x-y$ remains away from the origin and so $\lambda_v(x-y)$ remains bounded. 
\par 
If $x$ is near the origin, then we know that $|\lambda_v(x-y) - \log |x-y|| \leq O(1)$. But $\left| \int \log |x-y| \delta_{\zeta(0,M^{-1})} \right| \leq \log M$ and so 
$$\left| \int \lambda_v(x-y) \delta_{\zeta(0,\eps)} \right| \leq \log M + O(1) = O(\log M)$$
as desired. To bound the Lipschitz constant, it suffices to prove that for $x,x'$ with $|x-x'|_v$ sufficiently small, we have $\left| \lambda_{v,M}(x) - \lambda_{v,M}(x') \right| \leq O(M)|x-x'|_v$. This is equivalent to bounding 
$$\left|\int \lambda_v(x-y) - \lambda_v(x'-y) (\delta_{\zeta(0,M^{-1})}) \right|.$$
If $x$ is not near the origin, then as $\lambda_v$ is a smooth function, its derivative is bounded and so $\left| \lambda_v(x-y) - \lambda_v(x'-y) \right| \leq O(|x-x'|_v)$ for $y \in \supp \delta_{\zeta(0,M^{-1})}$ and so we are done. Else we know that $\lambda_v(x-y) - \log |x-y|_v$ is a smooth function, say $h_v$, and so 
$$\left| \int (\lambda_v(x-y) - \lambda_v(x'-y)) (\delta_{\zeta(0,M^{-1})}) \right| $$
$$\leq \left|\int (h_v(x) - h_v(x')) \delta_{\zeta(0,M^{-1})} \right| + \left|\int (\log |x-y| - \log |x'-y|) \delta_{\zeta(0,M^{-1})} \right|.$$
Again as $h_v(x)$ is smooth, we may bound the first integral by $O(|x-x'|_v)$. The second expression may be computed to be $\max(-\log M, \log |x|) - \max(-\log M, \log |x'|)$. This may be bounded by $MO(|x-x'|_v)$ and so we conclude that $\Lip(\lambda_{v,M}) \leq O(M)$ as desired.
\end{proof}

Putting everything together, we obtain the following quantitative equidistribution bound for $\lambda_{v}(x-\beta)$. 

\begin{proposition} \label{ArchQuantLogBound1}
Let $F$ be a $\Gal(\ovl{K}/L)$-invariant set of torsion points for $E$ and let $C = \max_{x \in F} \lambda_v(x-\beta)$. Let $N$ be the number of elements of $F$ such that $\lambda_v(x-\beta) \geq 4 \log F$. Then for $|F| \geq O(D)$, we have 
$$\frac{1}{[L:\bb{Q}]}\left|\sum_{x \in F} \lambda_v(x-\beta) \right| \leq \frac{N(C+ \log |F|)}{|F|} + O\left( \frac{\log |F|}{|F|^{1/2}} \right).$$
\end{proposition}

\begin{proof}
Take $M = |F|^2$. By Proposition \ref{ArchQuant2}, we know that 
$$\frac{1}{[L:\bb{Q}]} \frac{1}{|F|} \sum_{x \in F} \left( \lambda_{v,M}(x) - \lambda_v(x-\beta) \right) \leq O(|F|^{-1}) + \frac{N(C + \log |F|)}{|F|}.$$
Applying Proposition \ref{ArchQuant1} along with Proposition \ref{ArchQuant2} to bound $\langle \lambda_{v,M}, \lambda_{v,M} \rangle$ and $\Lip(\lambda_{v,M})$, as $|F| \geq O(D)$ we obtain 
$$\frac{1}{[L:\bb{Q}]} \left|\frac{1}{|F|} \sum_{x \in F} \lambda_{v,M}(x) - \int \lambda_{v,M}(x) \right| \leq O\left(\frac{\log |F|}{|F|^{1/2}} \right).$$
Finally since $\lambda_{v,M}(x)$ differs from $\lambda_v(x-\beta)$ by at most $O(|F|^{-1})$ outside a disc of radius $|F|^{-1}$ around $\beta$ and the integral of $\lambda_{v,M}(x)$ and $\lambda_v(x-\beta)$ over this disc is at most $O(\frac{\log |F|}{|F|})$, we get 
$$\frac{1}{[L:\bb{Q}]}\left|\frac{1}{|F|} \sum_{x\in F} \lambda_v(x-\beta) - \int \lambda_v(x-\beta) d \mu_v \right| \leq \frac{N(C+ \log |F|)}{|F|} + O\left(\frac{\log |F|}{|F|^{1/2}}\right)$$
which gives us our Proposition as $\int \lambda_v(x-\beta) d \mu_v = 0$. 
\end{proof}

\section{Quantitative Logarithmic Equidistribution for Elliptic Curves}

\subsection{Bounding distances to torsion points for $E$}
We now move onto the case of elliptic curves. We will require a version of linear forms on logarithms that handle points on elliptic curves. Let $E$ be an elliptic curve and fix an isomorphism $\bb{C}/\Lambda \simeq E(\bb{C})$, where $\Lambda$ is the period lattice. This map is given by $z \mapsto [\wp(z) : \wp'(z) : 1]$ where $\wp(z)$ is the Weierstrass p-function associated to the lattice $\Lambda$. If we let $\omega_1,\omega_2$ be a basis for the lattice $\Lambda$ and $\Pi$ a fundamental domain for $\bb{C}/\Lambda$, then we may view our isomorphism $\Pi \simeq E(\bb{C})$ as an exponential map and we denote the inverse map by $\log_E$. We now state a theorem due to David--Hirata-Kohno with some simplifications done. For an elliptic curve $E$, we let $h_E$ denote the Neron--Tate height on it.

\begin{theorem}[Theorem 1, \cite{DHK02}] \label{EllipticBaker1}
Let $k$ be a positive integer and let $E$ be an elliptic curve defined over some number field $K$. There exists a constant $C > 0$ such that the following hold: Let $K'$ be a number field of degree $D$ over $K$ and $L(z) = \beta_0 z_0 + \cdots + \beta_k z_k$ be a non-zero linear form on $\bb{C}^{k+1}$ with coefficients in $K'$. Let $u_1,\ldots,u_k$ be complex numbers such that $\gamma_i = (1,\wp(u_i), \wp'(u_i)) \in E(K')$. Let $B,V_1,\ldots,V_k$ be real numbers satisfying
$$\log B \geq \max\{1, h(\beta_i); 0 \leq i \leq k \}$$
$$V_1 \geq \cdots \geq V_k$$
$$\log V_i \geq \max \left \{ e, h_{E}(\gamma_i), \frac{|u_i|^2}{D} \right \}.$$
Then if $L(v) \not = 0$ for $v = (1,u_1,\ldots,u_k)$, we have
$$\log |L(v)| \geq -C D^{2k+2} (\log B + \log(eD) + \log \log V_1)(\log (eD) + \log \log V_1)^{k+1} \prod_{i=1}^{k} (1 + \log V_i).$$
\end{theorem}

\begin{corollary} \label{EllipticBaker2}
Let $E$ be an elliptic curve defined over a given number field $K$. Then there exists a constant $C > 0$ for which the following holds: Let $L/K$ be an extension of $K$ with degree $D$. Let $b_1,b_2$ a rational numbers and $\alpha$ an element of $E(L)$. Let 
$$\Lambda =  b_1  + b_2 \omega + \log_E \alpha$$
where our elliptic curve $E$ is given by the period lattice spanned by $\{1, \omega\}$. Let $B > 0$ be a constant such that $\log B > \max\{1,h(b_1),h(b_2)\}$. Then if $\Lambda \not = 0$, we have
$$\log |\Lambda| \geq -C D^6 (\log D + 1)^2 \log B (h_E(\alpha)+1) (\log^+ h_E(\alpha)+1)^4.$$
\end{corollary}

\begin{proof}
We apply Theorem \ref{EllipticBaker1} with $k = 2$. Then $\log V_1$ may be taken as $O(1)$, and $\log V_2$ taken as $O(h_E(\alpha))$ as $|u_i|$'s are bounded. The result then follows immediately. 
\end{proof}

Observe that there is an extra $\log^+ h_E(\alpha)$ factor in Corollary \ref{EllipticBaker2}, as compared to Corollary \ref{BakerCorollary1}. This factor is the main obstruction in proving a uniform result like Theorem \ref{IntroUniformIh1} for elliptic curves in general. For elliptic curves with complex multiplication, Ably and Graudon \cite{AG03} have managed to remove the $\log^+ h_E(\alpha)$ factor, which will allow us to prove Theorem \ref{IntroUniformIh2}. We now state a special case of their theorem that we need, rephrased in quantities that we use in Theorem \ref{EllipticBaker1}. 

\begin{theorem}[Theorem 0.1 \cite{AG03}] \label{EllipticBakerCM1}
Let $k$ be a positive integer and let $E$ be an elliptic curve with complex multiplication that is defined over some number field $K$. There exists a constant $C > 0$ such that the following hold: Let $K'$ be a number field of degree $D$ over $K$ and $L(z) = \beta_0 z_0 + \cdots + \beta_k z_k$ be a non-zero linear form on $\bb{C}^{k+1}$ with coefficients in $K'$. Let $u_1,\ldots,u_k$ be complex numbers such that $\gamma_i = (1,\wp(u_i), \wp'(u_i)) \in E(K')$. Let $B,V_1,\ldots,V_k$ be real numbers satisfying
$$\log B \geq \max\{1, h(\beta_i); 0 \leq i \leq k \}$$
$$V_1 \geq \cdots \geq V_k$$
$$\log V_i \geq \max \{ e, h_{E}(\gamma_i), \frac{|u_i|^2}{D} \}.$$
Then if $L(v) \not = 0$ for $v = (1,u_1,\ldots,u_k)$, we have
$$\log |L(v)| \geq -C (D^n (1 + D \log (1+D) \prod_{i=1}^{n} \log V_i))$$
$$\times (D \log B + \log (e + \max_{1 \leq i \leq n} |u_i| + \log \max_{1 \leq i \leq n} \{1 , \frac{1}{|u_i|} \})).$$
\end{theorem}

\begin{corollary} \label{EllipticBakerCM2}
Let $E$ be an elliptic curve with complex multiplication that is defined over a number field $K$ and let $D$ be a positive integer. Then there exists a constant $C$ such that the following holds: Let $b_1,b_2$ be two rational numbers and $\alpha$ an element of $E(L)$ where $[L:K] \leq D$. Let 
$$\Lambda =  b_1  + b_2 \omega + \log_E \alpha$$
where our elliptic curve $E$ is given by the period lattice spanned by $\{1, \omega\}$. Let $B > 0$ be a constant such that $\log B > \max\{1,h(b_1),h(b_2)\}$. Then if $\Lambda \not = 0$, we have
$$\log |\Lambda| \geq -C (h_E(\alpha)+1) \log B.$$
\end{corollary}

\begin{proof}
This follows immediately from Theorem \ref{EllipticBakerCM1}, where now since we fix an upper bound of the degree $D$ and do not care about the dependency on $D$, we can treat $D$ as a constant. Then we may choose $|u_i|$ to be some bounded domain away from the origin, so that both $|u_i|$ and $\frac{1}{|u_i|}$ are bounded. We are then left with the $\prod_{i=1}^{n} \log V_i$ term, which can be chosen to be $O(\log_E h(\alpha))$ as desired.
\end{proof}

As an immediate corollary, we obtain the following. 

\begin{proposition} \label{EllipticCMBound1}
Let $E$ be an elliptic curve with complex multiplication, $v$ an archimedean place of $K$ and let $D$ be a positive integer. Then for any $\epsilon > 0$, there exists a constant $C_{\epsilon}$ such that for any $\beta \in E(\ovl{K})$ with $\deg_K(\beta) \leq D$ and any $\Gal(\ovl{K}/K)$-orbit $F$ of torsion points, we have
$$\max_{x \in F} \lambda_v(x-\beta) \leq C_{\epsilon} (h_{E}(\beta)+1)|F|^{\epsilon}$$
\end{proposition}

\begin{proof}
Let $n$ be the order of $x$ as a torsion point. Then $\log_E x$ can be written as $\frac{a_1}{n} + \frac{a_2}{n} \omega$ for some positive integers $a_1,a_2$. Applying Corollary \ref{EllipticBakerCM2}, we obtain that
$$\log |\log_E \beta - \log_E x|^{-1}_v < C (h(\alpha)+1) \log n.$$
Now since $|F| > \frac{cn}{(\log \log n)^2}$ by say (45) and (46) of \cite{BIR08}, and we can replace $\log n$ with $|F|^{\epsilon}$. Finally, by Proposition \ref{ArchRegBound1}, we know that 
$$\left| \lambda_v(x-y) - \log |\log_E \beta - \log_E x |^{-1}_v  \right|\leq O(1)$$
and so we obtain the corresponding bound for $\lambda_v(x,\beta)$ too. 
\end{proof}

For a general elliptic curve $E$, applying Theorem \ref{EllipticBaker1} directly is not enough to prove the bound we need, as there is an extra factor of $\log^+ h(\alpha)$. Instead, we will exploit the fact that torsion points cannot be too close to each other. 
\begin{proposition} \label{EllipticBound1}
Let $E$ be an elliptic curve over a number field $K$ and let $v$ be an archimedean place of $K$ that is extended to $\ovl{K}$. Then for any $\epsilon > 0$, there exists a constant $C_{\epsilon}$ such that the following holds: for all extensions $L/K$ with $[L:K] \leq D$ and $\beta \in E(L)$, we have
$$\max_{x \in F} \lambda_v(x-\beta) \leq C_{\epsilon}D^6 (\log D)^2 (h(\beta)+1)|F|^{\epsilon}$$
for all $\Gal(\ovl{K}/L)$-orbits $F$ of torsion points, where $|F| \geq D$, with the possible exception of one.
\end{proposition}

\begin{proof}
By Corollary \ref{EllipticBaker2}, if we write $x$ as $b_1 \omega_1 + b_2 \omega_2$, then we have the bound
$$\max_{x \in F} \log |\log_E \beta - x|^{-1}_v < C (\log B) D^4 (\log D)^2 (h_E(\beta)+1) (\log^+ h_E(\beta)+1)^4.$$
where $\log B = \max\{h(b_1),h(b_2)\}$. If $n$ is the order of $x$ as a torsion point, then $\log B = \log n$. By Proposition \ref{EllipticGalois1}, we know that $|F|  > \frac{cn^{1-\eps}}{D}$ for some constant $c > 0$ and since $|F| \geq D$, it follows that $|F|^2 \geq cn^{1-\eps}$. Thus if $n > h_E(\beta)^{2/\epsilon}$, we would get
\begin{equation} \label{eq:Elliptic2}
\max_{x \in F} \log |\log_E\beta - x|^{-1}_v < C D^6 (\log D)^2 (h_E(\beta)+1) |F|^{\epsilon}
\end{equation}
for an appropriate constant $C$. Now let's say we have two torsion points $x_1,x_2$ that do not satisfy the inequality (\ref{eq:Elliptic2}). Thus we certainly have
$$ \log | \log_E \beta - \log_E x_i|^{-1}_v > C (h_E(\beta)+1) |F_i|^{\epsilon}$$
where $F_i$ is the Galois orbit of $x_i$. Assuming that $|F_1| \leq |F_2|$. by triangle inequality we have that
$$|\log_E x_1 - \log_E x_2| \leq |\log_E x_1 - \log_E \beta| + |\log_E \alpha - \log_E x_2| \leq 2e^{-C (h_E(\beta)+1)|F_1|^{\epsilon}}.$$
It follows that if $x_1 = b_{1,1} \omega_1 + b_{2,1} \omega_2$ and $x_2 = b_{1,2} \omega_1 + b_{2,2}\omega_2$, then the lowest common multiple of the denominators of $b_{i,j}$ have to be at least $O(e^{-C h_E(\beta) |F_1|})$. In particular as $|F_2| \geq |F_1|$ we must have $\max \{ h(b_{1,2}), h(b_{2,2}) \} \geq \frac{C}{4} h_E(\alpha) |F_1|$. Increasing our constant $C$, we can guarantee that
$$|F_2| \geq \frac{1}{\frac{C}{4}h_E(\beta)} e^{\frac{C}{4} h_E(\beta)} \geq h_E(\beta)^{2/\epsilon}.$$
This contradicts the upper bound in (\ref{eq:Elliptic2}). Thus for a given $\beta$, we can find a constant $C > 0$ such that 
$$\max_{x \in F} \log |\log_E \beta - \log _E x |^{-1}_v < C D^6 (\log D)^2 (h(\beta)+1)^2 |F|^{\epsilon},$$
is true for all Galois orbit $F$ of torsion points, with the possibility of one exception. Similar to Proposition \ref{EllipticCMBound1}, we then deduce the corresponding statement for $\lambda_v(x-\beta)$ as desired.
\end{proof}
 
We now handle the non-archimedean places. This will follow from a similar argument to the case of $\bb{G}_m$, where we proved that no two roots of unity can be $p$-adically close to each other. First, we prove an upper bound on $\lambda_v(x)$ for any torsion point $x$. 

\begin{proposition} \label{NonArchimedeanBound1}
Let $E$ be an elliptic curve defined over a number field $K$ and $v$ a non-archimedean place. Then there exists $\delta > 0$, depending on $v$, such that for any $\beta \in E(K_v)$, there does not exist two torsion points $z_1,z_2$ satisfying $\lambda_v(\beta - z_i) \geq \delta$. 
\end{proposition}

\begin{proof}
In Corollary 3.6 of \cite{BIR08}, such a constant $\delta$ is proven when $\beta$ is fixed. It suffices to note that their proof does not depend on $\beta$. In the good reduction case, if we let $Nv$ denote the size of the residue field at $v$, they deduce that $||z,0||_v \leq (Nv)^{-\delta}$ where $z = z_1 - z_2$. This implies that if $y^2 = x^3 + Ax + B$ is a minimal Weierstrass model for $E/K_v$, then $\ord_v(x(z)), \ord_v(y(z)) \geq \frac{\delta}{\log Nv}$ which is a contradiction to Cassel's theorem for $\delta$ large enough. 
\par 
In the case where $E(\bb{C}_v)$ is a Tate curve, it is proven again that $$\ord_v(x(z)), \ord_v(y(z)) \geq \frac{\delta}{\log Nv},$$ 
which again contradicts Cassel's theorem for $\delta$ large enough. 
\end{proof}

Now similar to Proposition \ref{PowerBound2}, we will deduce the following proposition.

\begin{proposition} \label{NonArchimedeanBound2}
Let $K$ be a number field, $v$ a non-archimedean place along with an extension to $\ovl{\bb{Q}}$, and $D$ a positive integer. Given an elliptic curve $E$ defined over $K$ with semistable reduction, there exists constants $C,\delta > 0$ depending on $D$ such that for any $\beta \in E(L)$ with $[L:K] < D$ and torsion point $z$, we have
$$\lambda_v(z-\beta) \geq \delta + \lambda_v(\beta)$$
if $\deg_K(z) > C$. 
\end{proposition}

\begin{proof}
First, since there are only finitely many extensions of $K_v$ of degree $\leq D$, there exists a single finite extension $K'_w$ of $K_v$ such that $v$ embeds $L$ into $K'_w$ for any finite extension $L$ of degree $\leq D$. 
\par 
Let $E/K_v$ have a minimal Weierstrass model given by the equation $y^2 = x^3 + Ax + B$. Then by \cite[Section 3.2]{BP05}, we know that if $\lambda_v(z-\beta) > \frac{1}{12} \log^+|j_E|_v$, then $z-\beta \in E_0(\bb{C}_v)$ and 
$$\lambda_v(z-\beta) = \frac{1}{2} \log^+|x(z-\beta)|_v^{-1} + \frac{1}{12} \log^+|j_E|_v.$$
In particular taking $\delta$ as large as we want, we may assume that $\log|x(z-\beta)|_v \geq \delta' + \log |x(\beta)|_v$ for any $\delta' > 0$. Now by the addition formula, we know that 
$$x(z-\beta) = \left(\frac{y(\beta)-y(z)}{x(\beta)-x(z)} \right)^2 - x(z) + x(\beta).$$
In particular as $|y(z)|_v, |x(z)|_v$ are both bounded and $|y(\beta)|_v = |x(\beta)|_v^{3/2}$ if $|x(\beta)|_v$ is large enough, we would have 
$$|x(z-\beta)|_v \leq |x(\beta)|_v $$
for all $\beta$ with sufficiently large $|x(\beta)|_v$. This is a contradiction and hence $|x(\beta)|_v$ and thus $|y(\beta)|_v$ must both be bounded from above. For any fixed $\eps > 0$, taking $\delta'$ large enough implies that $|x(\beta) - x(z)|_v \leq \eps$.
\par 
Now let $M$ be an upper bound for $|x(\beta)|_v$.  Then we can cover the elements of $\bb{P}^1(K'_w)$ with valuation $\leq C$ with finitely many discs of radius $\eps$. For sufficiently small $\eps$, by Cassel's theorem there can be at most one torsion point such that $x(z)$ is in $D_i$. Since there are finitely many such discs, we may choose $\delta$ large enough so that $\lambda_v(z-\beta) \leq \delta$ for all $\beta$ satisfying $|x(\beta)|_v \leq M$. This then concludes our proof. 
\end{proof}

\subsection{Proof of Theorems \ref{IntroUniformIh2} and \ref{IntroUniformIh4}}

We now prove Theorems \ref{IntroUniformIh2} and \ref{IntroUniformIh4} in a similar fashion to Theorems \ref{IntroUniformIh1} and \ref{IntroUniformIh3}.
\par 
We first start with Theorem \ref{IntroUniformIh2}. Let $D$ be a positive integer and let $\beta \in E(L)$ with $[L:K] \leq D$ be a non-torsion point. Let $S'$ be the set of places of $L$ that are above those of $S$. Given any torsion point $x$ with $\Gal(\ovl{K}/L)$-orbit $F$, we have the following formula 
$$\frac{1}{|F|}\sum_{v \in M_L} \sum_{x \in F} N_v \lambda_v(x-\beta) = h_E(x-\beta) = h_E(\beta)$$
as $h_E(x) = 0$. Now assume that $\beta$ is $S$-integral relative to $x$ and assume that $S$ contains all places of bad reduction for $E$. Then $\lambda_v(x-\beta) = 0$ for all $v \not \in S'$ and $x \in F$. Hence we have 
$$\frac{1}{|F|} \sum_{v \in S'} \sum_{x \in F} N_v \lambda_v(x-\beta) = h_E(\beta).$$

\begin{proof}[Proof of Theorem \ref{IntroUniformIh2}]
Let $N$ be the number of points in $F$ such that $\lambda_v(x-\beta) > 2 \log |F|$ and assume that $|F|$ is sufficiently large. By Proposition \ref{NonArchimedeanBound1}, when $v$ is non-archimedean we may take $N = 1$. When $v$ is archimedean, if $F$ consists of $m$-torsion points, we know that $|F| \geq O(\frac{m^{1/2}}{[L:\bb{Q}]})$ by Proposition \ref{EllipticGalois1}. Then there is at most $1$ torsion point that have distance $O(|F|^{-3})$ from $\beta$ as each $m$-torsion point is at least $\frac{c}{m}$ distance away for some $c > 0$. Hence again we may take $N = 1$.  
\par 
We wish to apply Proposition \ref{ArchQuantLogBound1} along with \eqref{eq: NonArchQuantLog1}. By Proposition \ref{EllipticCMBound1}, as $D$ is fixed, we may take $C$ to be $C_{\eps}(h_E(\beta)+1) |F|^{\eps}$ for archimedean places $v$ and by Proposition \ref{NonArchimedeanBound2}, we may take $C$ to be $\delta + O( h_E(\beta)) $ for non-archimedean places. As $D$ is fixed and $N \leq 1$, we obtain
\begin{equation} \label{eq: QuantLogBound1}
\left| \frac{1}{|F|} \sum_{v \in S'} \sum_{x \in F} N_v \lambda_v(x-\beta) \right| \leq O\left( \frac{h_E(\beta)+1}{|F|^{1-\eps}} + \frac{\log |F|}{|F|^{1/2-\eps}} \right).
\end{equation}
Now taking $|F|$ sufficiently large, as $\int \lambda_v(x-\beta) d \mu_v = 0$, we obtain 
$$h_E(\beta) = \frac{1}{|F|} \sum_{v \in S'} \sum_{x \in F} N_v \lambda_v(x-\beta) \leq \frac{\eps (h_E(\beta)+1)}{D^4}$$
for any desired $\eps > 0$. But we know that $h_E(\beta) = O(\frac{1}{D^{3+\eps}})$ by \cite{Mas89} which gives a contradiction. 
\end{proof}

\begin{proof}[Proof of Theorem \ref{IntroUniformIh4}] Let $D = [L:K]$. We may assume that our $\Gal(\ovl{K}/L)$-orbit of torsion points $F$ satisfies $|F| \geq D$. By Propositions \ref{EllipticBound1} and \ref{NonArchimedeanBound1}, outside of possibly $|S|$ $\Gal(\ovl{K}/L)$-orbits of torsion points as exceptions, we may take $C = C_{\eps}D^{6 + \eps} (h_E(\beta)+1)|F|^{\eps}$ as an upper bound for $\max_{x \in F} \lambda_v(x-\beta)$ for both archimedean and non-archimedean places. 
\par 
Hence by Proposition \ref{ArchQuantLogBound1} and \eqref{eq: NonArchQuantLog1}, we get
$$ h_E(\beta) = \frac{1}{|F|} \sum_{v \in S} \sum_{x \in F} N_v \lambda(x-\beta) $$
$$\leq |S||F|^{-1+\eps} O_{\eps}\left( D^{7+\eps} (h_E(\beta)+1) \right) + O_{\eps}\left( D |S| |F|^{-1/2+ \eps}\right).$$
Taking $|F| \geq c |S|^3 D^{11}$ for sufficiently large $c$ gives us
$$c'h_E(\beta) \leq  D^{-3-\eps} (h_E(\beta)+1)$$
for any $c' > 0$ we wish. But we know that $h_E(\beta) \geq O\left(D^{-3+\eps} \right)$ by \cite{Mas89} which gives us a contradiction as desired. 
\end{proof}

\printbibliography

\end{document}